\documentclass[a4paper, 10pt]{article}
\usepackage[T2A]{fontenc}
\usepackage[english]{babel}

\usepackage{amsmath, amsthm}
\usepackage{amssymb}
\usepackage{amsfonts}
\usepackage{srcltx, dsfont, color, graphicx}

\theoremstyle{plain}
\newtheorem{theorem}{Theorem}[section]
\newtheorem{lemma}{Lemma}[section]

\numberwithin{equation}{section}

\def\tht{\theta}
\def\Om{\Omega}
\def\om{\omega}
\def\e{\varepsilon}

\def\G{\Gamma}
\def\l{\lambda}
\def\p{\partial}
\def\D{\Delta}
\def\a{\alpha}

\def\d{\delta}

\def\vk{\varkappa}

\def\Ho{\mathring{W}_2}

\def\hf{\mathfrak{h}}

\def\la{\langle}
\def\ra{\rangle}

\def\cL{\mathcal{L}}

\def\Ho{\mathring{W}}

\DeclareMathOperator{\RE}{Re}

\DeclareMathOperator{\mes}{mes}

\DeclareMathOperator{\dist}{dist}
\DeclareMathOperator{\supp}{supp}

\allowdisplaybreaks

\begin{document}

\title{Operator estimates for non-periodically perforated domains with Dirichlet and nonlinear Robin conditions: vanishing limit}

\author{D.I. Borisov$^{1,2,3}$\footnote{Corresponding author}, J. K\v{r}\'{\i}\v{z}$^3$}

\date{\empty}

\maketitle

{\small
    \begin{quote}
1) Institute of Mathematics, Ufa Federal Research Center, Russian Academy of Sciences,  Chernyshevsky str. 112, Ufa, Russia, 450008
\\
2) Bashkir State  University, Zaki Validi str. 32, Ufa, Russia, 450076
\\
3) University of Hradec Kr\'alov\'e
62, Rokitansk\'eho, Hradec Kr\'alov\'e 50003, Czech Republic
\\
Emails: borisovdi@yandex.ru, jan.kriz@uhk.cz
\end{quote}

{\small
    \begin{quote}
    \noindent{\bf Abstract.} We consider a general second order linear elliptic equation in a finely perforated domain. The shapes of cavities  and their distribution in the domain are arbitrary and non-periodic; they are supposed to satisfy minimal natural geometric conditions. On the boundaries of the cavities we impose either the Dirichlet or a nonlinear Robin condition; the choice of the type of the boundary condition for each cavity is arbitrary. Then we suppose that for some cavities the nonlinear Robin condition is sign-definite in certain sense. Provided such cavities and ones with the Dirichlet condition are distributed rather densely in the domain and the characteristic sizes of the cavities and the minimal distances between the cavities satisfy certain simple condition, we show that a solution to our problem tends to zero as the perforation becomes finer. Our main result are order sharp estimates for the $L_2$- and $W_2^1$-norms of the solution uniform in the $L_2$-norm of the right hand side in the equation.

    \medskip

    \noindent{\bf Keywords:} perforated domain, non-periodic perforation, operator estimates, vanishing limit, order sharp estimates

    \medskip

    \noindent{\bf Mathematics Subject Classification: 35B27, 35B40}
 	
    \end{quote}
}

\section{Introduction}

Elliptic boundary value problems in finely perforated domains are one of the classical objects in the modern homogenization theory.  A typical formulation of such problems consists in an elliptic equation in a domain perforated by small closely spaced cavities with some boundary conditions on their boundaries. The main aim of the study is to characterize the behavior of the solutions to the considered problems as the perforation becomes finer. There are hundreds of works devoted to studying such problems and not trying to mention all of them, we just cite several books, where problems in perforated domains were considered \cite{Sha}, \cite{ZKO}, \cite{MK}, \cite{MK1}, \cite{OIS}, see also the references therein. Typical results obtained for the problems in perforated domains usually state a convergence of the solutions to those of some homogenized problems. The convergence is usually proved in $L_2$ or in $W_2^1$, in a weak or strong sense, for given right hand sides in the equations and boundary conditions.

During the last 20 years, a new direction is being developed in the homogenization theory devoted to so-called operator estimates. The matter is that the difference between the solutions to the perturbed and homogenized problems is estimated uniformly with respect to the $L_2$-norm of the right hand side in the equation. From the point of view of the spectral theory, in the case of linear equations such estimates correspond to the norm resolvent convergence and describe the convergence rates, while the classical results usually establishes the strong or weak resolvent convergence.

Operator estimates for problems in periodically perforated domains were established in   \cite{Khra}, \cite{Khra2}, \cite{ChDR}, \cite{Sus}, \cite{Past}, \cite{Zhi}. In first two papers on the boundaries of the cavities  the Neumann condition was imposed. In \cite{Zhi}, the perforation was considered as an example of a general approach on homogenization of abstract measures, and it also led to the Neumann conditions on the boundaries of the cavities. In all three papers  \cite{Sus}, \cite{Past}, \cite{Zhi}, the sizes of the cavities were proportional to the distances between them. In \cite{Khra2}, the boundaries of the cavities were subject to the Dirichlet conditions, their linear sizes were assumed  to satisfy certain  upper bound expressed in terms of the size of the periodicity cell. All cavities had the same shape but it was allowed to rotate them arbitrarily and to shift a little their positions. In paper \cite{Khra}  the perforation was pure periodic and the cavities were small balls with the Robin condition on the their boundaries.  A similar setting was studied in
\cite{ChDR}, where on the boundaries of the cavities also the Dirichlet and Neumann conditions were imposed. In all these papers, the homogenized problems were found and various operator estimates were established. The sharpness of these estimates was not discussed.

The case of a non-periodic perforation was considered in \cite{Post}. Here the cavities were non-periodically distributed small balls with the Dirichlet or Neumann conditions on their boundaries and certain conditions were imposed on the sizes of these balls ensuring that under the homogenization  the balls disappear. In addition, there was considered a  case, when the sizes of the balls with the Dirichlet condition were not too small so that the solution vanished in the limit in a part of the domain covered by such balls. The main results in \cite{Post} were also operator estimates for the considered problems.

There are also a few papers on operator estimates in domains finely perforated along a given manifold \cite{PRSE16}, \cite{MSB21}, \cite{AA22}. Here the perforation was of a general non-periodic structure, both the shapes of cavities and their distribution were arbitrary. Various cases of possible homogenized problems were considered and in all of them operator estimates were established. In some cases these estimates were shown to be order sharp.

In the present paper we consider a boundary value problem for a general  linear second order elliptic equation in a perforated domain; the equation is not necessary formally symmetric and involves complex-valued coefficients. The perforation is of a general structure, namely, both the shapes and distribution of the holes are arbitrary and only minimal natural geometric conditions are imposed on the perforation. The minimal distance between the cavities is characterized by a small positive parameter $\e$, while a characteristic linear size of each cavity is $\e\eta$, where $\eta=\eta(\e)$ is some bounded function.

On the boundaries of the cavities we impose either the Dirichlet or a nonlinear Robin condition. The choice of type of the boundary condition for each cavity is arbitrary and in the general situation we have the mixing of boundary conditions, that is, the boundaries of some cavities are subject to the Dirichlet condition, while on the boundaries of other cavities the nonlinear Robin condition is imposed. We assume that on the boundaries of  some cavities, the nonlinear Robin condition is sign-definite in certain sense and its strength is controlled by some function $\mu=\mu(\e)\geqslant 1$. We also suppose that such cavities together with the Dirichlet cavities are distributed rather densely in the domain. These two assumptions and also certain condition for the relation between $\e$, $\eta$ and $\mu$ then ensure that a solution to the considered vanishes in the limit $\e\to+0$, that is, as the perforation becomes finer. Namely, we establish operator estimates both in $L_2$- and $W_2^1$-norms uniformly in $L_2$-norm of the right hand side in the equation writing explicitly the convergence rate in terms of $\e$, $\eta$ and $\mu$. The convergence rate for the $L_2$-norm of the solution is twice better than that for the $W_2^1$-norm. We succeed to show that the found convergence rates  are order sharp; this is adduced by appropriate examples. One more important feature of our results is that they are true under minimal assumptions for the perforation,  which are expressed in geometric terms. They allow us to establish key local estimates, which can be interpreted as  the spectral  solidifying condition, which was one of the main assumptions in \cite{Post}.

\section{Problem and main results}\label{sec2}

Let $x=(x_1,\ldots,x_n)$ be Cartesian coordinates in $\mathds{R}^n$ and $\Om$ be a domain in $\mathds{R}^n$, the boundary of which has the smoothness  $C^2$. The domain $\Om$ can be both bounded or unbounded; a particular case $\Om=\mathds{R}^n$ is also possible. In the paper, we study an elliptic boundary value problem in a domain obtained by perforating $\Om$. This perforation and the problem are introduced as follows.

By $\e$ we denote a small positive parameter. We choose a family of points $M_k^\e\in\Om$, $k\in\mathds{M}^\e$, and a family of domains $\om_{k,\e}\subset \mathds{R}^d$, $k\in\mathds{M}^\e$, where $\mathds{M}^\e$ is   some  set of indices, which is at most countable. The domains $\om_{k,\e}$ are bounded and have $C^1$-boundaries. We then denote:
\begin{equation}\label{2.4}
\om_k^\e:=\big\{x:\, (x - M_k^\e)\e^{-1}\eta^{-1}(\e)\in \om_{k,\e}\big\}, \quad k\in\mathds{M}^\e,\qquad \tht^\e:=\bigcup\limits_{k\in\mathds{M}^\e} \om_k^\e,
\end{equation}
where $\eta=\eta(\e)$ is some function such that $0<\eta(\e)\leqslant 1$. The aforementioned perforated domain is introduced as $\Om^\e:=\Om\setminus\tht^\e$.  We shall formulate rigorous conditions for the geometry of perforation later, now we just say that the cavities $\om_k^\e$ are disjoint and at the same time the distances between the points $M_k^\e$ are small.

The coefficients of the elliptic equation we shall study are complex-valued functions $A_{ij}^\e=A_{ij}^\e(x)$, $A_j^\e=A_j^\e(x)$, $A_0^\e=A_0^\e(x)$ defined in the perforated domain $\Om^\e$, which are supposed to satisfy the conditions
\begin{equation}\label{2.5}
A_{ij}^\e,\, A_j^\e,  A_0^\e\in L_\infty(\Om^\e),
\qquad A_{ij}^\e=\overline{A_{ji}^\e},\qquad \sum\limits_{i,j=1}^{n} A_{ij}^\e(x)\xi_i\overline{\xi_j}\geqslant c_0\sum\limits_{j=1}^{n} |\xi_j|^2,\qquad x\in\Om^\e,\quad \xi_i\in\mathds{C},
\end{equation}
where $c_0>0$ is some fixed constant independent of $\e$, $\xi$ and $x$. 
The functions $A_{ij}^\e$, $A_j^\e$, $A_0^\e$ are bounded uniformly in $\e$ in the norm of $L_\infty(\Om^\e)$.

We introduce an arbitrary partition of the set
$\tht^\e$:
\begin{equation}\label{2.8a}
\tht_D^\e:=\bigcup\limits_{k\in\mathds{M}_D^\e} \om_k^\e,\qquad \tht_R^\e:=\bigcup\limits_{k\in\mathds{M}_R^\e} \om_k^\e, \qquad \mathds{M}_D^\e\cup \mathds{M}_R^\e=\mathds{M}^\e, \qquad \mathds{M}_D^\e\cap \mathds{M}_R^\e=\emptyset.
\end{equation}
By $a^\e=a^\e(x,u)$ we denote a measurable complex-valued function defined for $x\in\p\tht_R^\e$ and $u\in\mathds{C}$ 
 obeying the Lipschitz conditions
\begin{equation}\label{2.6}
|a^\e(x,u_1)-a^\e(x,u_2)|\leqslant a_0^\e |u_1-u_2|,\qquad u_1,u_2\in \mathds{C},
\end{equation}
where $a_0^\e$ is some constant independent of $x$, $u_1$, $u_2$ but depending on $\e$.  We also assume that the function $a^\e(x,u)$ satisfies the estimate
\begin{equation}
\label{2.11}
\RE a^\e(x,u)\overline{u}\geqslant -c_1 \e \eta^{-n+1}(\e)
|u|^2,\qquad x\in\p\tht_R^\e,
\end{equation}
where $c_1$ is some fixed constant independent of $\e$ and $\eta$;  this constant can be   positive or negative or zero.

In this paper we study the following boundary value problem:
\begin{equation}\label{2.7}
(\cL-\l)u_\e=f\quad\text{in}\quad \Om^\e,\qquad u_\e=0\quad\text{on} \quad \p\Om\cup\p\tht_D^\e,\qquad  \frac{\p u_\e}{\p\boldsymbol{\nu}} + a^\e(x,u_\e)=0\quad\text{on}\quad \p\tht_R^\e.
\end{equation}
Here $\cL$ is a differential expression:
\begin{equation}\label{2.8}
\cL:=-\sum\limits_{i,j=1}^{n} \frac{\p\ }{\p x_i} A_{ij}^\e \frac{\p\ }{\p x_j}  + \sum\limits_{j=1}^{n} A_j\frac{\p\ }{\p x_j}  + A_0,
\end{equation}
$f\in L_2(\Om^\e)$ is an arbitrary function, $\l\in\mathds{C}$ is a fixed constant
and the co-normal derivative is defined as
\begin{equation*}
\frac{\p\ }{\p\boldsymbol{\nu}} = \sum\limits_{i,j=1}^{n} A_{ij}^\e \nu_i \frac{\p\ }{\p x_j},
\end{equation*}
where $\nu_i$ are the components of the unit normal to $\p\tht^\e_R$ directed inside $\tht_R^\e$.

The main aim of our study is to describe the behavior of solution to problem (\ref{2.7})  as $\e\to+0$. A solution is understood in the generalized sense. Given a domain $Q\subseteq\mathds{R}^n$ and a manifold $S\subset \overline{Q}$ of codimension one, let $\Ho_2^1(Q,S)$ be the subspace of the Sobolev space $W_2^1(Q)$ formed by the functions with the zero trace on the manifold $S$. A generalized solution to problem (\ref{2.7})
 is  a function $u\in \Ho_2^1(\Om^\e,\p\Om\cup\p\tht_D^\e)$  satisfying the integral identity
\begin{equation}\label{2.18}
\hf_a(u_\e,v)-\l(u_\e,v)_{L_2(\Om^\e)}=(f,v)_{L_2(\Om^\e)}
\end{equation}
for each $v\in \Ho_2^1(\Om^\e,\p\Om\cup\p\tht_D^\e)$, where
\begin{align*}
&\hf_a(u,v):=\hf_A(u,v) + (a^\e(\,\cdot\,,u),v)_{L_2(\p\tht_R^\e)},
\\
&
\hf_A(u,v):= \sum\limits_{i,j=1}^{n} \left(A_{ij}^\e \frac{\p u}{\p x_j}, \frac{\p u}{\p x_i}\right)_{L_2(\Om^\e)} + \sum\limits_{j=1}^{n} \left(A_j^\e \frac{\p u}{\p x_j}, v\right)_{L_2(\Om^\e)}
  + (A_0^\e u,v)_{L_2(\Om^\e)}.
\end{align*}

Now we are going to formulate exact conditions on the geometry of perforation. By  $B_r(M)$ we denote a ball in $\mathds{R}^n$ of a radius $r$ centered at a point $M$. In the vicinity of the boundaries $\p\om_{k,\e}$ we introduce a local variable, which is the distance $\tau$ measured along the normal to $\p\om_{k,\e}$ directed outside $\om_{k,\e}$. Our first assumption concerns the general structure of the perforation.

\begin{enumerate}
\def\theenumi{{A\arabic{enumi}}}
\item\label{A1} The points $M_k^\e$ and the domains $\om_{k,\e}$ obey the conditions
\begin{equation}
B_{R_1}(y_{k,\e})\subseteq \om_{k,\e}\subseteq B_{R_2}(0),
\quad
B_{\e R_3}(M_k^\e)\cap B_{\e R_3}(M_j^\e)=\emptyset, \quad
 \dist(M_k^\e,\p\Om)\geqslant R_3\e,\quad
 k\ne j,\quad k,j\in\mathds{M}^\e,\label{2.2}
\end{equation}
where $y_{k,\e}$ are some points  and $R_1<R_2<R_3$ are some fixed constants independent of $\e$, $\eta$, $k$ and $j$.
The sets $B_{R_2}(0)\setminus \om_{k,\e}$ are connected. For each $k\in\mathds{M}_R^\e$ there exist local variables $s$ on $\p\om_{k,\e}$ such that the variables $(\tau,s)$ are  well-defined at least on $\{x\in\mathds{R}^n\setminus\om_{k,\e}:\, \dist(x,\p\om_{k,\e})\leqslant \tau_0\} \subseteq B_{R_2}(0)$, where $\tau_0$ is a fixed constant independent of $k\in\mathds{M}^\e$ and $\e$ and the Jacobians corresponding to passing from variables $x$ to $(\tau,s)$ are separated from zero and bounded from above uniformly in
 $\e$, $k\in\mathds{M}^\e$ and $x$ as $0\leqslant \tau\leqslant \tau_0$. The first derivatives of $x$ in $(\tau,s)$ and of $(\tau,s)$ in $x$ are continuous and bounded uniformly  in
 $\e$, $k\in\mathds{M}^\e$ and $x$ as $0\leqslant \tau\leqslant \tau_0$.
\end{enumerate}

It is known that the behavior of a solution to problem (\ref{2.7}), (\ref{2.8}) is very sensible to the geometry of perforation and the structure of the coefficients in the equation and boundary conditions. In this paper we consider just one of several possible typical cases, namely, the case when $u_\e$ vanishes as $\e\to+0$. Such situation is ensured by the following additional conditions.

\begin{enumerate}
\def\theenumi{{A\arabic{enumi}}}\setcounter{enumi}{1}

\item\label{A6} The set $\mathds{M}_{R}^\e$ contains a subset $\mathds{M}_{R,0}^\e$ such that
\begin{equation}
 \RE a^\e(x,u)\overline{u}\geqslant \mu(\e) \a_k^\e(x)|u|^2, \qquad x\in\p\om_k^\e,\quad k\in\mathds{M}_{R,0}^\e,\label{2.25}
\end{equation}
where $\mu=\mu(\e)\geqslant 1$ and $\a_k^\e=\a_k^\e(x)$
are some measurable functions obeying the conditions
\begin{equation}\label{2.13a}
\begin{aligned}
&\a_k^\e\in L_\infty(\p\om_k^\e),\qquad \a_k^\e\geqslant 0\quad\text{a.e. on $\p\om_k^\e$},
\\
&\|\a_k^\e\|_{L_2(\p\om_k^\e)}^2\leqslant  c_2  (\e\eta)^{n-1},\qquad \|\a_k^\e\|_{L_1(\p\om_k^\e)}\geqslant  c_3  (\e\eta)^{n-1},
\end{aligned}
\end{equation}
where  $c_2$, $c_3$ are some fixed positive constants independent of $k\in\mathds{M}_{R,0}^\e$, $\e$ and $\eta$.

\item\label{A5} There exists a fixed constant $R_4>0$ independent of $\e$ and $k$ such that
\begin{equation}\label{2.21}
\Om\subseteq\bigcup\limits_{k\in\mathds{M}_{R,0}^\e\cup \mathds{M}_D^\e}
B_{\e R_4}(M_k^\e).
\end{equation}
\end{enumerate}

We denote:
\begin{equation*}
\vk(\e):=|\ln\eta(\e)|+1\quad\text{as}\quad n=2,\qquad
\vk(\e):=1\quad\text{as}\quad n\geqslant 3.
\end{equation*}

Now we are in position to formulate our main result.

\begin{theorem}\label{th1}
Let Conditions~\ref{A1},~\ref{A6},~\ref{A5} be satisfied
 and
\begin{equation}\label{2.14}
\e \eta^{-n+1}(\e) \mu^{-1}(\e) + \e^2\eta^{-n+2}(\e)\vk(\e) \to+0,\qquad \e\to+0.
\end{equation}
Then there exists a fixed $\l_0\in\mathds{R}$ independent of $\e$ such that as $\RE\l\leqslant \l_0$, problem (\ref{2.7}), (\ref{2.8})  is solvable for each $f\in L_2(\Om^\e)$ and each of its solutions satisfies the estimates:
\begin{align}\label{2.16}
&\|u_\e\|_{W_2^1(\Om^\e)}\leqslant C\big(\e^\frac{1}{2} \eta^{-\frac{n}{2}+\frac{1}{2}}(\e)\mu^{-\frac{1}{2}}(\e) + \e \eta^{-\frac{n}{2}+1}(\e)\vk(\e)\big)
 \|f\|_{L_2(\Om^\e)},
\\
\label{2.17}
&\|u_\e\|_{L_2(\Om^\e)}\leqslant C \big(\e \eta^{-n+1}(\e)\mu^{-1}(\e) + \e^2\eta^{-n+2}(\e)\vk(\e)\big) \|f\|_{L_2(\Om^\e)},
\end{align}
where $C$ is some constant independent of $\e$ and $f$. These estimates are order sharp.
\end{theorem}

\begin{theorem}\label{th2}
Let Assumptions~\ref{A1},~\ref{A5} be satisfied, the set $\mathds{M}_{R,0}^\e$ be empty,  and
\begin{equation}\label{2.31}
  \e^2\eta^{-n+2}(\e)\vk(\e) \to+0,\qquad \e\to+0.
\end{equation}
Then there exists a fixed $\l_0\in\mathds{R}$ independent of $\e$ such that as $\RE\l\leqslant \l_0$, problem (\ref{2.7}), (\ref{2.8})  is solvable for each $f\in L_2(\Om^\e)$ and each of its solutions satisfies the estimates:
\begin{align}\label{2.29}
&\|u_\e\|_{W_2^1(\Om^\e)}\leqslant C  \e \eta^{-\frac{n}{2}+1}(\e) \vk^{\frac{1}{2}}(\e)
 \|f\|_{L_2(\Om^\e)},
\\
\label{2.30}
&\|u_\e\|_{L_2(\Om^\e)}\leqslant C   \e^2\eta^{-n+2}(\e) \vk (\e) \|f\|_{L_2(\Om^\e)},
\end{align}
where $C$ is some constant independent of $\e$ and $f$. These estimates are order sharp.
\end{theorem}

Let us discuss briefly  the problem and main result. The first main feature is that the perforation we consider is of a general non-periodic structure. We only make  Assumption~\ref{A1}, which is quite natural. The first two-sided inclusion in (\ref{2.2}) means that the domains $\om_{k,\e}$ are roughly of the same sizes: each of these domain can be put inside a fixed ball $B_{R_2}(0)$ and inside each of these domain a fixed ball of the radius $R_1$ can be inscribed, see Figure~\ref{fig1}.  We stress that this does not impose any restrictions on the shapes of the domains $\om_{k,\e}$ and they indeed can be very arbitrary. In view of the inequality $R_2<R_3$, the second condition in (\ref{2.2}) means that the cavities $\om_k^\e$, defined in (\ref{2.4}), do not intersect and there is a minimal distance $2R_3$ between the points $M_k^\e$,
which can be regarded as a distance between the cavities. The third inequality in (\ref{2.2}) ensures that the cavities do not intersect the boundary of the domain $\p\Om$, which is also a natural assumption, see Figure~\ref{fig2}.  We note that these conditions do not imply that the distances between the points $M_k^\e$ are indeed of order $\sim \e$ since only the lower bounds for the distances are postulated. Some of the distances can be much larger than $\e$. It should also be said that the distribution of the points $M_k^\e$, and as a consequence, of the cavities $\om_k^\e$, is very weakly restricted by (\ref{2.2})
and it can be  very arbitrary.

\begin{figure}
\begin{center}
\includegraphics[scale=0.25]{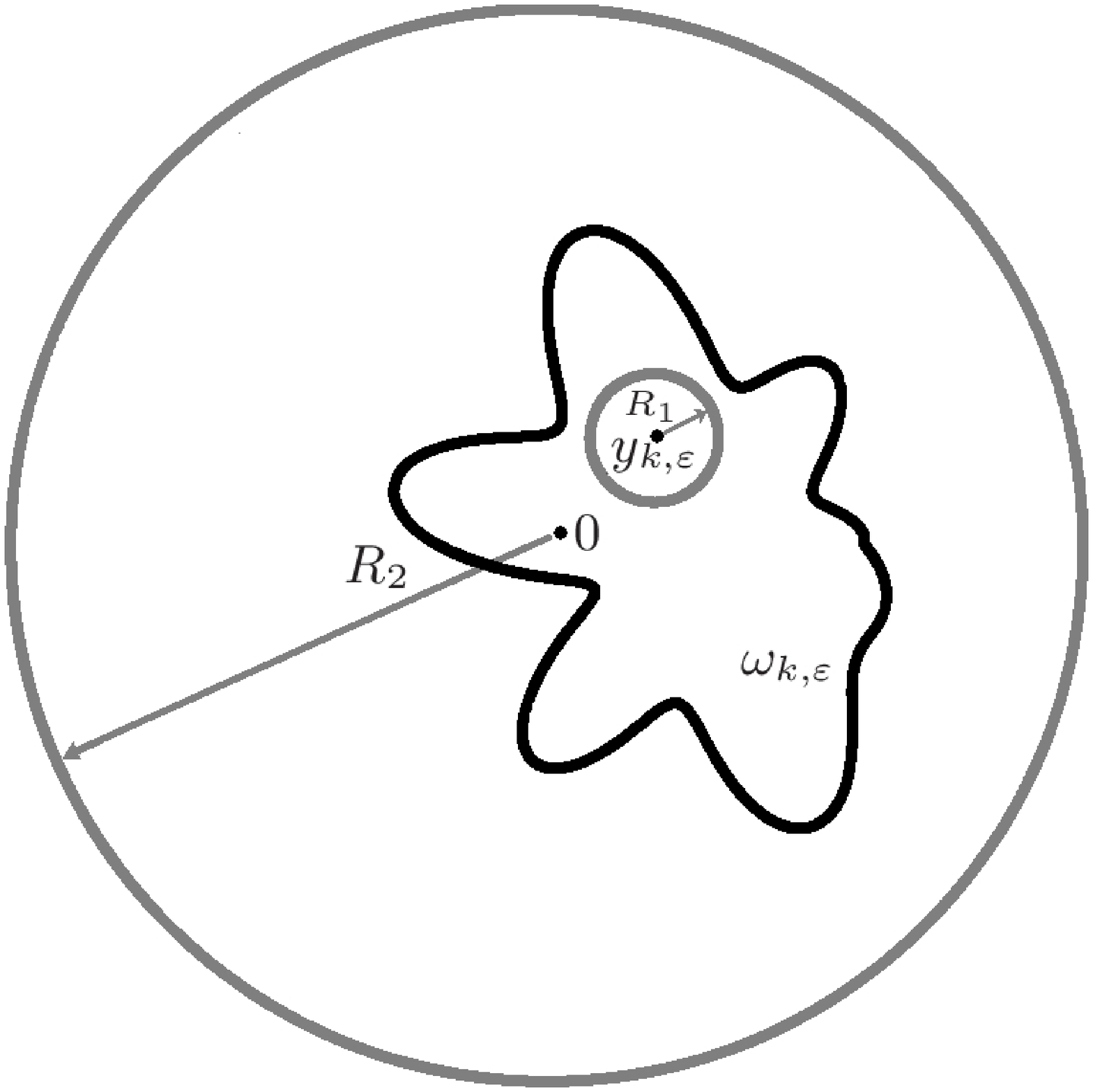}
\end{center}

\caption{Domain $\om_{k,\e}$ and corresponding balls $B_{R_1}(y_{k,\e})$ and $B_{R_2}(0)$}\label{fig1}
\end{figure}

The connectedness of sets $B_{R_2}(0)\setminus \om_{k,\e}$ is also a natural assumption.  Indeed, if this is not the case for some $k$, then the part $B_{\e\eta R_2(M_k^\e)}\setminus \om_k^\e$ of the domain $\Om^\e$ 
contains a small piece not connected with the rest of $\Om^\e$. Then the boundary value problem on this small piece becomes independent on the problem on the rest of $\Om^\e$ and
its solution do not influence the behavior of the solution on the rest of $\Om^\e$.
The assumption on the local variables in the vicinity of the boundaries $\p\om_{k,\e}$ for $k\in\mathds{M}_{R,0}^\e$ is more gentle and requires a certain uniform regularity of the shapes of boundaries $\p\om_{k,\e}$ with respect to $k$ and $\e$.  First, since the assumed smoothness of $\p\om_{k,\e}$ is $C^1$, there is a tangential hyperplane to $\p\om_{k,\e}$ at each point and the normal vector to this hyperplane is continuous. We additionally assume that the introduced local variables possess uniformly bounded continuous derivatives with respect to the initial Cartesian coordinates and the corresponding Jacobians are well-defined. Second, the uniform boundedness of these Jacobians is an additional regularity property. Roughly speaking, this regularity prohibits the situation when on some sequence of values $k$ and $\e$, the boundaries $\p\om_{k,\e}$ highly oscillates or approximates some non-smooth surfaces. We also conjecture that if this is the case, this could influence and change the behaviour of solutions to problem~(\ref{2.7}).

Problem~(\ref{2.7}) involves general linear differential expression $\cL$ defined in (\ref{2.8}), which is not formally symmetric due to the presence of the first derivatives and due to the complex coefficients. The coefficients in this expression can also depend on the parameter $\e$ and the only restrictions on the dependence are ones in (\ref{2.5}). These are very weak conditions and this is why the dependence of the coefficients on $\e$ can be very arbitrary including, for instance, a particular case of fast and non-periodically oscillating functions.

The Dirichlet boundary condition on $\p\Om$ is imposed just for the definiteness and can be replaced by any other classical boundary condition. The main feature are the conditions on the boundaries of the cavities: these are the Dirichlet and the nonlinear Robin conditions. The choice of type of the boundary conditions is in fact arbitrary and is  described by the partition of cavities in (\ref{2.8a}). The nonlinearity in the Robin condition should satisfy Lipschitz condition (\ref{2.6}), which means that it has at most a linear growth. One more condition is inequality (\ref{2.11}), which restricts the growth of the nonlinearity with respect to the small parameter $\e$.

Extra conditions are Assumptions~\ref{A6},~\ref{A5}. The former  states that at least for some cavities with the nonlinear Robin condition the nonlinearity is positive definite in the sense of inequality (\ref{2.25}), where the functions $\a_k^\e$ should obey (\ref{2.13a}). The estimates for the $L_1$- and $L_2$-norms of these functions can be equivalently rewritten as
\begin{equation*}
\|\a_{k,\e}\|_{L_2(\p\om_{k,\e})}^2\leqslant c_2,\qquad \|\a_{k,\e}\|_{L_1(\p\om_{k,\e})}\geqslant c_3,\qquad \a_{k,\e}(x):=\a_k^\e(M_k^\e+\e\eta x).
\end{equation*}
Here the spaces $L_1$ and $L_2$ are considered on non-small surfaces $\p\om_{k,\e}$ and these inequalities mean that the functions $\a_{k,\e}$ should be bounded in $L_2(\om_{k,\e})$-norms uniformly in $k$ and $\e$ and should have an $L_1(\om_{k,\e})$-norm uniformly separated from zero. It is clear that these conditions are very weak and are obeyed by a very wide class of functions $\a_{k,\e}$. The parameter $\mu=\mu(\e)$ in (\ref{2.25}) characterizes the dependence of the nonlinearity $a^\e$ on $\p\om_k^\e$ with respect to $\e$. This parameter behaves arbitrarily as $\e\to+0$, including a possible arbitrary growth.

\begin{figure}
\begin{center}
\includegraphics[scale=0.25]{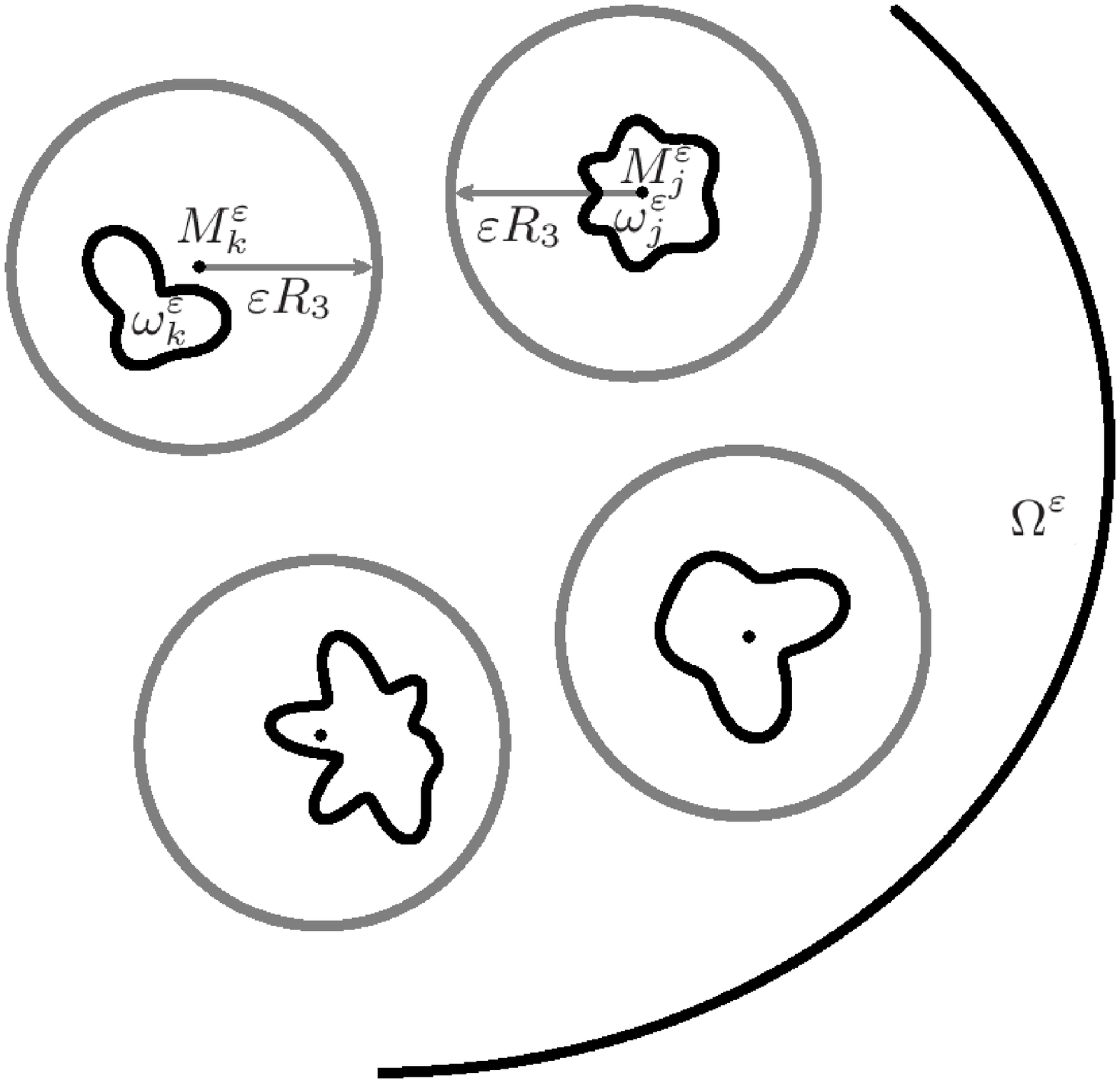}
\end{center}
\caption{Perforated domain $\Om^\e$ and cavities $\om_k^\e$}\label{fig2}
\end{figure}

Assumption~\ref{A5} has a very simple geometric interpretation: the cavities with the Dirichlet condition and the nonlinear Robin condition obeying Assumption~\ref{A6}   should be distributed quite dense to ensure  covering (\ref{2.21}) of the domain $\Om$, see Figure~\ref{fig3}. This assumption is used at a final step in the proofs of Theorems~\ref{th1},~\ref{th2}. Namely, we first make a series of certain local estimates for the solution $u_\e$ in balls $B_{\e R_4}(M_k^\e)$ and then we glue these estimates into global inequalities (\ref{2.16}), (\ref{2.17}), (\ref{2.29}), (\ref{2.30}).
Assumption~\ref{A6} allow us to cover in this way the entire domain $\Om^\e$; while doing this, by Assumption~\ref{A1} we also confirm that each point $x$ of $\Om^\e$ is covered by finitely many balls  $B_{\e R_4}(M_k^\e)$ and the number of such balls for each point $x$ is bounded uniformly in $\e$ and $x$.

Our main result states that provided convergence (\ref{2.14}) holds, problem~(\ref{2.7}) is solvable and each solution  tends to zero as $\e\to+0$. The convergence is established both in $W_2^1$- and $L_2$-norms and it is uniform in the $L_2$-norm of the right hand side in the equation.  A gentle point is  that we do not state a \textsl{unique} solvability and the situation of multiple solutions is not excluded. In the proof of Theorem~\ref{th1}, the solvability is ensured by Assumption~\ref{th1} and inequalities (\ref{2.6}), (\ref{2.11}); other assumptions and conditions are not used. Assumptions~\ref{A6},~\ref{A5} and convergence (\ref{2.14}) guarantee inequalities (\ref{2.16}), (\ref{2.17}). The first of them is for the $W_2^1(\Om^\e)$-norm  of the solution and it describes the rate, at which the solution vanishes as $\e\to+0$ uniformly in  the $L_2(\Om^\e)$-norm of the right hand side. The second inequality is of the same nature but now the $L_2(\Om^\e)$-norm of the solution is estimated. The second norm of the solution  is weaker than the norm in (\ref{2.16}) and the convergence rate is twice better.

If there is no cavities with the nonlinear Robin conditions obeying Assumption~\ref{A6} but the cavities with the Dirichlet condition still satisfy Assumption~\ref{A5}, this case can be also treated and the result is formulated in Theorem~\ref{th2}. It is similar to that in Theorem~\ref{th1}, but now convergence (\ref{2.14}) is replaced by (\ref{2.31}) and this is reflected in the convergence rates in (\ref{2.29}), (\ref{2.30}).

An  important feature of our main results is that the estimates in Theorems~\ref{th1},~\ref{th2} are order sharp. This is justified by appropriate examples provided in the proofs of these theorems.

\section{Proofs}

In this section we prove Theorems~\ref{th1},~\ref{th2}.

 \subsection{Auxiliary lemmata}

Here we prove a series of auxiliary lemmata, which play a crucial role in the proofs of Theorems~\ref{th1},~\ref{th2}.

\begin{lemma}\label{lm3.4}
Under Assumption~\ref{A1}
for each $k\in \mathds{M}_D^\e$ and each function $u\in \Ho^1(B_{\e R_3}(M_k^\e)\setminus \om_k^\e, \p\om_k^\e)$   the estimate
\begin{equation*}
\|u\|_{L_2(B_{\e R_3}(M_k^\e)\setminus \om_{k,\e})}^2 \leqslant C\e^2 \eta^{-n+2}\vk  \|\nabla u \|_{L_2(B_{\e R_3}(M_k^\e)\setminus \om_{k,\e})}^2
\end{equation*}
holds,
where $C$ is a constant independent of $k$, $\e$, $\eta$ and $u$.
\end{lemma}

\begin{proof}
Given a function $u\in \Ho^1(B_{\e R_3}(M_k^\e)\setminus \om_k^\e, \p\om_k^\e)$, we continue it by zero inside $\om_k^\e$. Then the continuation is an element of $\Ho^1(B_{\e R_3}(M_k^\e))$, its $L_2(B_{\e R_3}(M_k^\e))$-norm and $W_2^1(B_{\e R_3}(M_k^\e))$-norm coincide with
$L_2(B_{\e R_3}(M_k^\e)\setminus \om_k^\e)$-norm and $W_2^1(B_{\e R_3}(M_k^\e)\setminus \om_k^\e)$-norm of the original function $u$, and in view of the first condition in (\ref{A1}), the continuation vanishes at least on the ball $B_{\e\eta R_1}(M_k^\e+\e\eta y_{k,\e})$. We keep the same notation $u$ for this continuation. By $r$ we denote the radius for spherical coordinates centered at the point $M_k^\e+\e\eta y_{k,\e}$. Then for all $x\in B_{\e R_3}(M_k^\e)\setminus B_{\e\eta R_1}(M_k^\e+\e\eta y_{k,\e})$ by the Cauchy-Schwarz inequality we have:
\begin{align*}
|u(x)|^2=&\Bigg|\int\limits_{|x-M_k^\e -\e\eta y_{k,\e}|}^{\e\eta R_1} \frac{\p u}{\p r}\,dr\Bigg|^2
\leqslant \int\limits_{\e\eta R_1}^{|x-M_k^\e -\e\eta y_{k,\e}|} \frac{dr}{r^{n-1}}
 \int\limits_{\e\eta R_1}^{|x-M_k^\e -\e\eta y_{k,\e}|}  |\nabla u|^2 r^{n-1}\,dr
 \\
 \leqslant &
 C(\e\eta)^{-n+2}\vk \int\limits_{\e\eta R_1}^{|x-M_k^\e -\e\eta y_{k,\e}|}  |\nabla u|^2 r^{n-1}\,dr,
\end{align*}
where $C$ is some constant independent of $\e$, $\eta$, $k$ and $u$. Integrating this estimate over $B_{\e R_3}(M_k^\e)\setminus B_{\e\eta R_1}(M_k^\e+\e\eta y_{k,\e})$, we arrive at the statement of the lemma. The proof is complete.
\end{proof}

\begin{figure}
\begin{center}
\includegraphics[scale=0.25]{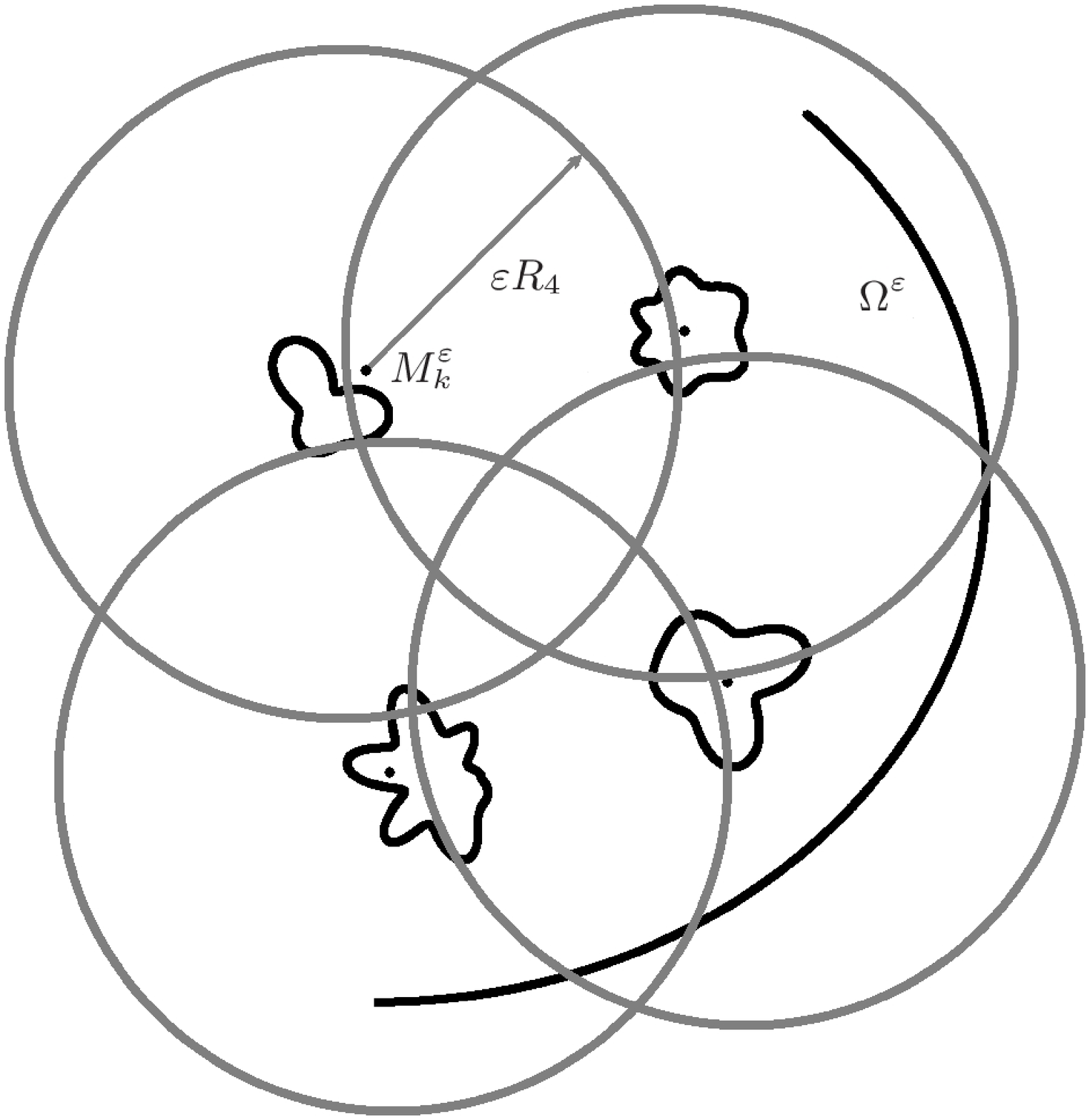}
\end{center}
\caption{Covering of  $\Om^\e$ by balls  $B_{\e R_4}(M_k^\e)$}\label{fig3}
\end{figure}

The next lemma was stated in \cite[Lm. 3.1]{MSB21} and its proof was essentially based on the Cheeger's estimate for the lowest positive eigenvalue of a Laplacian in a given domain
provided in the original work \cite{Cheeger}. However, as we found very recently, the Cheeger's estimate is to be modified appropriately once we deal with the Neumann condition on some parts of the boundary, see, for instance, \cite[Introduction]{HH}, \cite[Ch. 9, Cor. 9.7]{Li}. Despite the found gap in the proof of Lemma~3.1 in \cite{MSB21}, the lemma is still true and here we give a corrected proof which is not based on the Cheeger's estimate.

\begin{lemma}\label{lm3.0}
Under Assumption~\ref{A1}, for each $k\in \mathds{M}^\e$ and each function $u\in \Ho_2^1(B_{R_3}(0)\setminus \om_{k,\e}, \p B_{R_3}(0))$ the estimate holds:
\begin{equation*}
\|u\|_{L_2(B_{R_3}(0)\setminus \om_{k,\e})}\leqslant C \|\nabla u\|_{L_2(B_{R_3}(0)\setminus \om_{k,\e})},
\end{equation*}
where $C$ is a constant independent of $\e$, $k$ and $u$.
\end{lemma}

\begin{proof}
We choose an arbitrary function $u\in\Ho^1(B_{R_3}(0)\setminus \om_{k,\e}, \p B_{R_3}(0))$. Since the lowest eigenvalue of the Laplacian in $B_{R_3}(0)\setminus B_{R_2}(0)$ subject to the Dirichlet condition on $\p B_{R_3}(0)$ and to the Neumann condition on $\p B_{R_2}(0)$ is strictly positive, by the minimax principle we immediately obtain:
\begin{equation}\label{3.42}
\|u\|_{L_2(B_{R_3}(0)\setminus B_{R_2}(0))}^2\leqslant C \|\nabla u\|_{L_2(B_{R_3}(0)\setminus B_{R_2}(0))}^2;
\end{equation}
throughout the proof by $C$ we denote various inessential constants independent of $\e$, $k$, $u$ and the spatial variables.

We fix a positive number $\tau_1\leqslant \frac{\tau_0}{10\sqrt{n}}$ and consider a lattice $\tau_1\mathds{Z}^n$ in $\mathds{R}^n$. By $\G_{k,\e}$ we denote a subset of this lattice defined as
\begin{equation*}
\G_{k,\e}:=\big\{z\in\tau_1\mathds{Z}^n:\ z+2\tau_1 (0,1)^n\subset B_{R_3}(0)\setminus\om_{k,\e}\big\}.
\end{equation*}
The number of the points in the set $\G_{k,\e}$ satisfies an obvious estimate
\begin{equation}\label{3.43}
\# \G_{k,\e}\leqslant \# \tau_1\mathds{Z}^n\cap B_{R_3}(0),
\end{equation}
which is uniform in $k$ and $\e$. It also follows from the definition of $\G_{k,\e}$ that choosing $\tau_1$ sufficiently small but fixed, we have the covering
\begin{equation}\label{3.44}
B_{\frac{R_2+R_3}{2}}(0)\setminus\tilde{\om}_{k,\e} \subseteq \bigcup\limits_{z\in\G_{k,\e}} (z+2\tau_1(0,1)^n),\qquad \tilde{\om}_{k,\e}:=\bigg\{x:\ \dist(x,\om_{k,\e})\leqslant \frac{\tau_0}{2}\bigg\},\qquad \om_{k,\e}\subset\tilde{\om}_{k,\e},
\end{equation}
for all $\e$ and $k$.
Since the domain $B_{R_3}(0)\setminus\om_{k,\e}$ is connected by Assumption~\ref{A1},  the above covering yields that for each $z\in \G_{k,\e}\cap B_{\frac{2R_2+R_3}{3}}(0)$ such that $\dist\big(\p\om_{k,\e},z+2\tau_1(0,1)^n\big)\geqslant \frac{\tau_0}{2}$ there exist points $y_j=y_j(z)\in\G_{k,\e}$, $j=0,\ldots,N(z)$,  $y_N(z)=z$ such that
\begin{equation}
y_0(z)+2\tau_1(0,1)^n\subset B_{\frac{R_2+R_3}{2}}(0)\setminus B_{R_2}(0),\qquad y_j(z)+2\tau_1(0,1)^n\subset B_{\frac{R_2+R_3}{2}}(0)\setminus\om_{k,\e},\qquad y_j\ne y_p,\quad j\ne p,
\label{3.45}
\end{equation}
and 
for each $j$ the coordinates of the  points $y_j$ and $y_{j+1}$ differ at most by $\tau_1$. Due to estimate (\ref{3.43}), the number $N(z)$ is bounded uniformly in $\e$, $k$ and $z$. For each two neighbouring points $y_j$ and $y_{j+1}$ the cubes $y_j+2\tau_1(0,1)^n$ and $y_{j+1}+2\tau_1(0,1)^n$ have a non-empty intersection, which contains at least an appropriate shift of the cube $\tau_1(0,1)^n$, see Figure~\ref{fig4}.

For each  $z\in \G_{k,\e}\cap B_{\frac{2R_2+R_3}{3}}(0)$ such that $\dist\big(\p\om_{k,\e},z+2\tau_1(0,1)^n\big)\geqslant \frac{\tau_0}{2}$ we choose a corresponding point $y_0(z)$ obeying (\ref{3.45}) and by (\ref{3.42}) we have the estimate
\begin{equation}\label{3.46}
\|u\|_{L_2(y+\tau_1(0,1)^n)}\leqslant C \|\nabla u\|_{L_2(B_{R_3}(0)\setminus\om_{k,\e})}
\end{equation}
with $y=y_0$. Then we choose corresponding points $y_j$ and suppose that for some $j$ estimate (\ref{3.46}) holds with $y=y_j$. Then the cubes $y_{j+1}+2\tau_1 (0,1)^n$  intersects with $y_j+2\tau_1(0,1)^n$ at least by a cube with side $\tau_1$,  see Figure~\ref{fig5}; we denote this cube by $Q_j$. At least one of the vertices of $Q_j$ coincide with one of the vertices of $y_{j+1}+2\tau_1 (0,1)^n$; we denote this vertex by $q_j$. For an arbitrary point $x\in y_{j+1}+2\tau_1 (0,1)^n$ separated from $q_j$ by a distance at least $\frac{\tau_1}{2}$, we consider a segment $I(x)$ connecting the point $q_j$ and $x$ with a variable $t\in[0,|x-q_j|]$ on this segment. Then for $x\in y_{j+1}+2\tau_1(0,1)^n$ such that $|x-q_j|>\frac{\tau_1}{2}$ we obviously have
\begin{equation*}
u(x)=\int\limits_{0}^{|x-q_j|} \frac{\p\ }{\p t} u\chi_{1}(t)\,dt,
\end{equation*}
where the integration is taken along the segment $I(x)$ and $\chi_{1}=\chi_{1}(t)$ is an infinitely differentiable function equalling to one as $t>\frac{\tau_1}{2}$ and vanishing as $t<\frac{\tau_1}{3}$.  Then by the Cauchy-Schwarz inequality we immediately get:
\begin{equation*}
|u(x)|^2\leqslant C \int\limits_{0}^{|x-q_j|}  |\nabla u|^2 t^{n-1}\,dt+C \int\limits_{0}^{\frac{\tau_1}{2}}|u|^2  t^{n-1}\,dt,
\end{equation*}
where the integration is again made along the segment $I(x)$. Integrating the obtained inequality over $(y_{j+1}+2\tau_1(0,1)^n)\setminus Q_j$, passing to the spherical coordinates and using (\ref{3.46}), we find:
\begin{equation*}
\|u\|_{L_2((y_{j+1}+2\tau_1(0,1)^n)\setminus Q_j)}^2\leqslant C \|\nabla u\|_{W_2^1(y_{j+1}+2\tau_1(0,1)^n)}^2+C\|u\|_{L_2(Q_j)}^2\leqslant C \|\nabla u\|_{L_2(B_{R_3}(0)\setminus\om_{k,\e})}^2
\end{equation*}
and this implies (\ref{3.46}) for $y=y_{j+1}$.  Hence, using the above described procedure of extending estimate (\ref{3.46}) along the points $y_j$ and summing up   all obtained estimates over $y\in  \G_{k,\e}\cap B_{\frac{2R_2+R_3}{3}}(0)$ such that
$$
\dist\big(\p\om_{k,\e},y+2\tau_1(0,1)^n\big)\geqslant \frac{\tau_0}{2},
$$
by (\ref{3.42}) we finally get:
\begin{equation}\label{3.47}
\|u\|_{L_2(B_{R_3}(0)\setminus \hat{\om}_{k,\e})}\leqslant C \|\nabla\|_{L_2(B_{R_3}(0)\setminus \hat{\om}_{k,\e})},\qquad \hat{\om}_{k,\e}:=\bigg\{x:\ \dist(x,\om_k^\e)\leqslant \frac{7\tau_0}{10}\bigg\}.
\end{equation}

Let $\chi_{2}=\chi_{2}(\tau)$ be an infinitely differentiable cut-off function equalling to one as $\tau<8\tau_0/10$ and vanishing as $\tau>9\tau_0/10$. Then for $x$ such that $\tau<7\tau_0/10$ due to the regularity of the boundaries $\p\om_{k,\e}$ postulated in Assumption~\ref{A1} we have
\begin{equation}\label{3.48}
|u(x)|^2=\Bigg|\int\limits_{\tau}^{\tau_0} \frac{\p\ }{\p\tau} u\chi_{2}\,d\tau\Bigg|^2
\leqslant C \int\limits_0^{\tau_0} (|\nabla u|^2+|u|^2)\,d\tau.
\end{equation}
Integrating this estimate over $\hat{\om}_{k,\e}\setminus\om_{k,\e}$ and using then (\ref{3.47}), we obtain:
\begin{equation*}
\|u\|_{L_2(\hat{\om}_{k,\e}\setminus\om_{k,\e})}^2\leqslant C\|\nabla u\|_{L_2(\hat{\om}_{k,\e}\setminus\om_{k,\e})}^2 + C \|u\|_{L_2(B_{R_3}(0)\setminus \hat{\om}_{k,\e})}^2  \leqslant C \|\nabla u\|_{L_2(B_{R_3}(0)\setminus \om_{k,\e})}^2.
 \end{equation*}
This estimate and (\ref{3.47}) imply the estimate in the statement of the lemma. The proof is complete.
\end{proof}

\begin{figure}
\begin{center}
\includegraphics[scale=0.3]{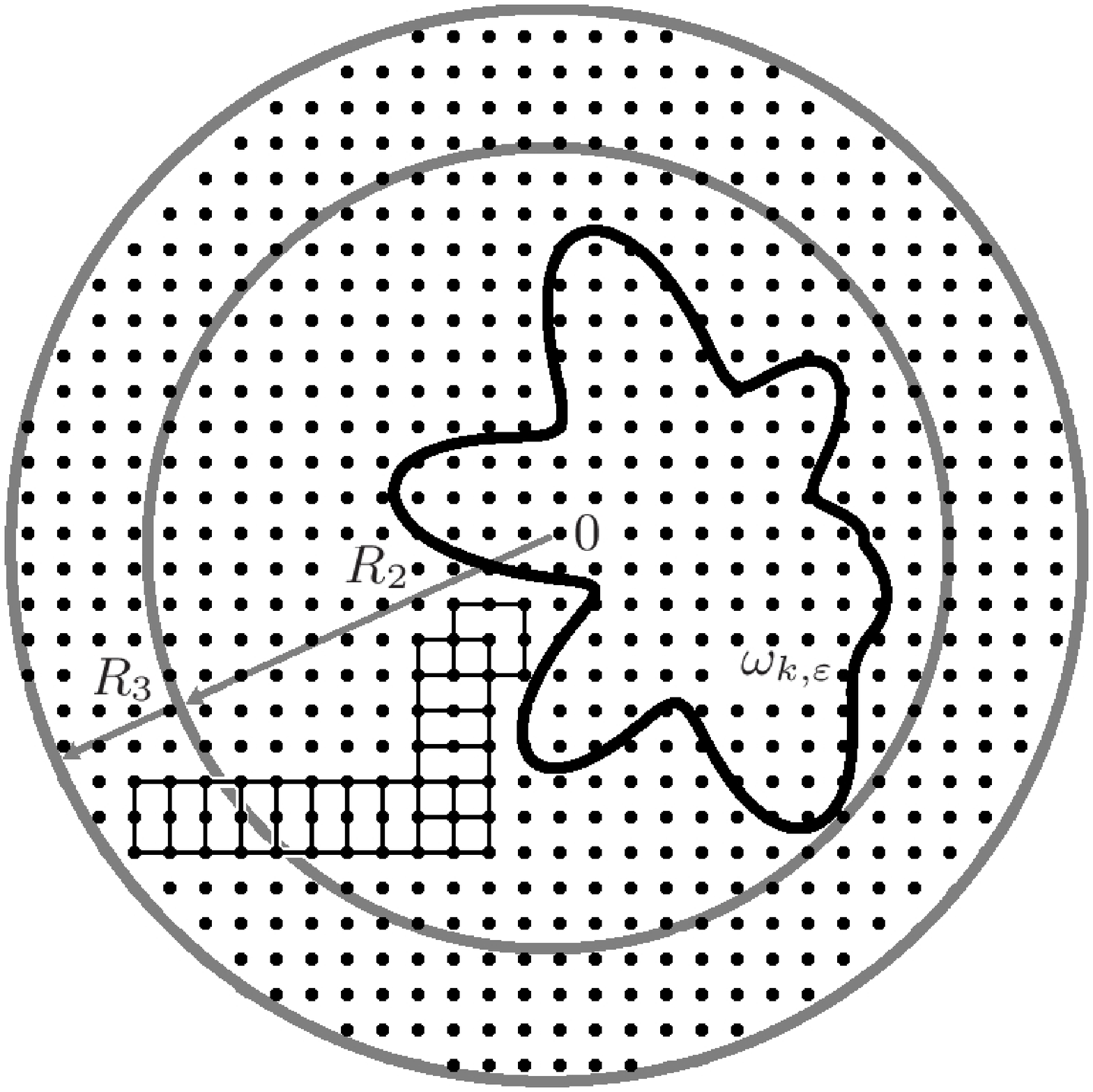}
\end{center}

\caption{Lattice points on $B_{R_3}(0)\setminus\om_{k,\e}$ and the path formed by the cubes $y_j+2\tau_1(0,1)^n$.}\label{fig4}
\end{figure}

\begin{lemma}\label{lm3.7}
Under Assumption~\ref{A1} the problem
\begin{equation}\label{3.19}
-\D X_{k,\e}=1\quad\text{in}\quad B_{R_3}(0)\setminus\overline{\om_{k,\e}},\qquad X_{k,\e}=0\quad\text{on}\quad \p B_{R_3}(0),\qquad \frac{\p X_{k,\e}}{\p\nu}=0\quad\text{on}\quad\p\om_{k,\e},\quad
\end{equation}
possesses a unique generalized solution for all $\e$ and $k\in\mathds{M}^\e$. This solution belongs to $W_2^1(B_{R_3}(0)\setminus\om_{k,\e})$ and satisfies the estimate
\begin{equation}\label{3.22}
\|X_{k,\e}\|_{W_2^1(B_{R_3}(0)\setminus\om_{k,\e})}  \leqslant C, \qquad \left\|\frac{\p X_{k,\e}}{\p |x|}\right\|_{L_2(\p B_{R_5}(0))}\leqslant C,\qquad  R_5:= \frac{R_2+R_3 }{2},
\end{equation}
with a constant  $C$ independent of $x$, $k$ and $\e$.
\end{lemma}

\begin{proof}
Due to the presence of the Dirichlet condition on $\p B_{R_3}(0)$ in problem (\ref{3.19}), it is uniquely solvable and the generalized solution belongs to $W_2^2(B_R(0)\setminus\om_{k,\e})$. The corresponding integral identity
\begin{equation*}
\|\nabla X_{k,\e}\|_{L_2(B_{R_3}(0)\setminus\om_{k,\e})}^2=\int\limits_{B_{R_3}(0)\setminus\om_{k,\e}} X_{k,\e}\,dx
\end{equation*}
and Lemma~\ref{lm3.0} by the Cauchy-Schwarz inequality give the estimates
\begin{equation*}
\|\nabla X_{k,\e}\|_{L_2(B_{R_3}(0)\setminus\om_{k,\e})}^2\leqslant C \|X_{k,\e}\|_{L_2(B_{R_3}(0)\setminus\om_{k,\e})}\leqslant C \|\nabla X_{k,\e}\|_{L_2(B_{R_3}(0)\setminus\om_{k,\e})}
\end{equation*}
with constants $C$ independent of $\e$ and $k$. These estimate imply the first inequality in (\ref{3.22}). Standard smoothness improving estimates then yield
\begin{equation*}
\|X_{k,\e}\|_{W_2^2(B_{R_5+\d}(0)\setminus B_{R_5-\d}(0))}\leqslant C
\|X_{k,\e}\|_{W_2^1(B_{R_5+2\d}(0)\setminus B_{R_5-2\d}(0))}\leqslant C
\end{equation*}
with $\d:=\frac{1}{5}(R_3-R_2)$ and a fixed constant $C$ depending only on $\d$. This yield the second inequality in (\ref{3.22}). The proof is complete.
\end{proof}

\begin{figure}
\begin{center}
\includegraphics[scale=0.5]{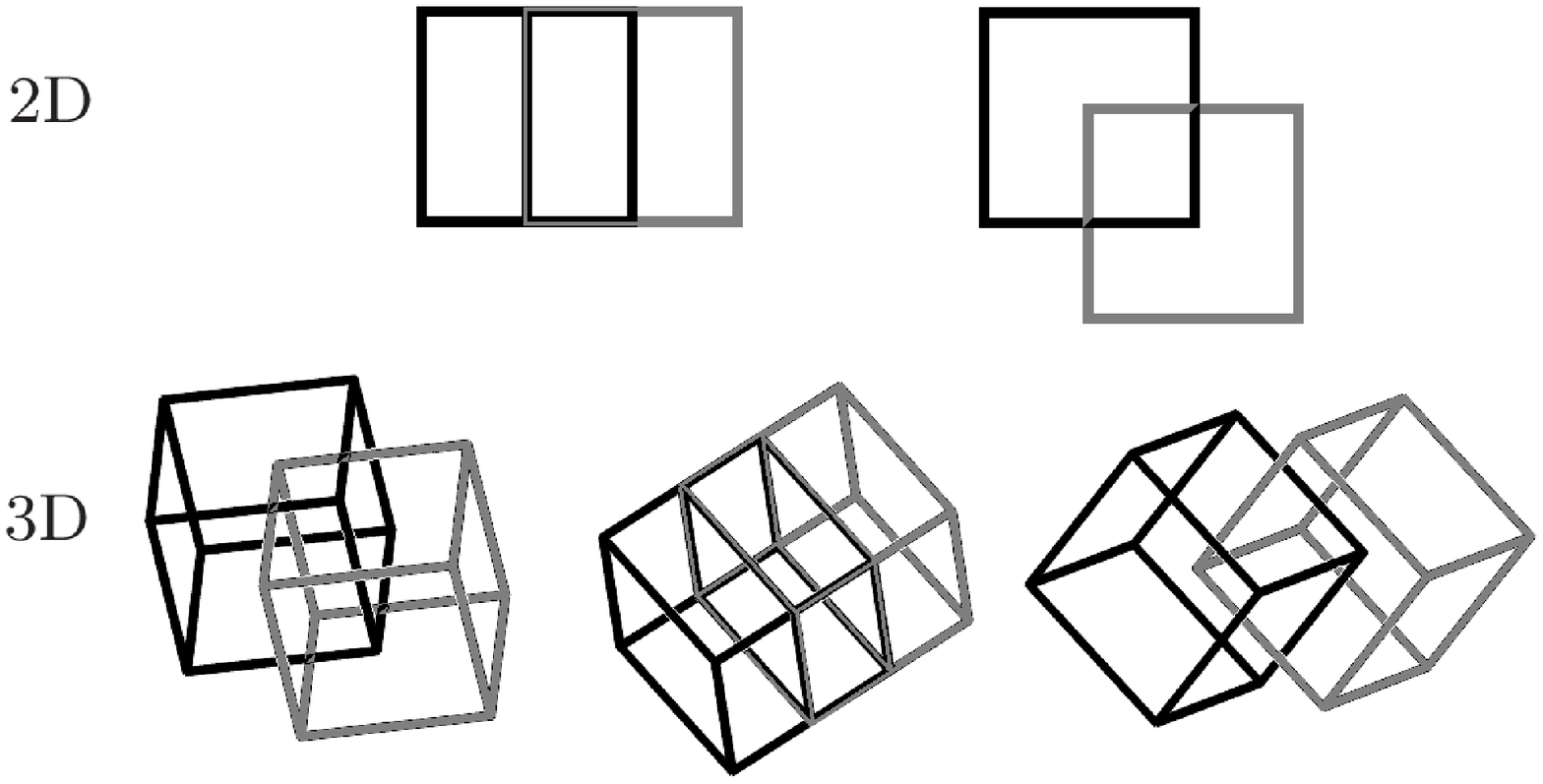}
\end{center}

\caption{Possible intersections of cubes $y_j+2\tau_1(0,1)^n$ in two- and three-dimensional cases.}
\label{fig5}
\end{figure}

\begin{lemma}\label{lm3.8}
Under Assumption~\ref{A1}, for each $k\in \mathds{M}_D^\e$ and each function $u\in W_2^1(B_{R_3}(0)\setminus \om_{k,\e})$ obeying the condition
\begin{equation}\label{3.24}
\int\limits_{B_{R_3}(0)\setminus \om_{k,\e}}u(x)\,dx=0
\end{equation}
the estimate holds:
\begin{equation}\label{3.26}
\|u\|_{L_2(B_{R_3}(0)\setminus \om_{k,\e})}\leqslant C \|\nabla u\|_{L_2(B_{R_3}(0)\setminus \om_{k,\e})},
\end{equation}
where $C$ is a constant independent of $\e$, $k$ and $u$.
\end{lemma}
\begin{proof}
Since the space $C^\infty(B_{R_3}(0)\setminus \om_{k,\e})$ is dense in $W_2^1(B_{R_3}(0)\setminus \om_{k,\e})$, it is sufficient to prove (\ref{3.26}) only for smooth functions $u$. Given such smooth function $u$, we denote
\begin{equation*}
\la u\ra:=\frac{1}{\mes_n B_{R_3}(0)\setminus B_{R_5}(0)}\int\limits_{B_{R_3}(0)\setminus B_{R_5}(0)} u(x)\,dx,\qquad R_5:=\frac{R_2+R_3}{2}, \qquad u_\bot:=u-\la u \ra,
\end{equation*}
where $\mes_n(\,\cdot\,)$ stands for the $n$-dimensional Lebesgue measure of a set. It is clear that
\begin{equation}
\|u\|_{L_2(B_{R_3}(0)\setminus B_{R_5}(0))}^2 = |\la u\ra|^2 \mes_n B_{R_3}(0)\setminus B_{R_5}(0) + \|u_\bot\|_{L_2(B_{R_3}(0)\setminus B_{R_5}(0))}^2,\quad \int\limits_{B_{R_3}(0)\setminus B_{R_5}(0)} u_\bot\,dx=0.\label{3.32}
\end{equation}
Since the second eigenvalue of the Neumann Laplacian in $B_{R_3}(0)\setminus B_{R_5}(0)$ is strictly positive, by the minimax principle the function $u_\bot$ satisfies the estimate
\begin{equation}\label{3.27}
\|u_\bot\|_{L_2(B_{R_3}(0)\setminus B_{R_5}(0))} \leqslant C \|\nabla u_\bot\|_{L_2(B_{R_3}(0)\setminus B_{R_5}(0))};
\end{equation}
throughout the proof by $C$ we denote various inessential constants independent of $u$,  $k$, $\e$ and spatial variables. This estimate then obviously implies that
\begin{equation}\label{3.28}
\|u_\bot\|_{L_2(\p B_{R_5}(0))}\leqslant C \|\nabla u\|_{L_2(B_{R_3}(0)\setminus B_{R_5}(0))}.
\end{equation}
We also observe that by (\ref{3.24})
\begin{equation}\label{3.30}
\la u\ra=-\frac{1}{\mes_n B_{R_3}(0)\setminus B_{R_5}(0)}\int\limits_{B_{R_5}(0)\setminus \om_{k,\e}} u\,dx.
\end{equation}

We take the solution  $X_{k,\e}$ to problem (\ref{3.19})   and integrate by partS:
\begin{equation*}
\mes_n B_{R_5}(0)\setminus \om_{k,\e}=-\int\limits_{B_{R_5}(0)\setminus \om_{k,\e}} \D X_{k,\e}\,dx=-\int\limits_{\p B_{R_5}(0)}\frac{\p X_{k,\e}}{\p |x|}\,ds.
\end{equation*}
Using this identity, we make a similar integration by parts:
\begin{align*}
\int\limits_{B_{R_5}(0)\setminus \om_{k,\e}} u\,dx = & -\int\limits_{B_{R_5}(0)\setminus \om_{k,\e}} u\D X_{k,\e}\,dx
=-\int\limits_{\p B_{R_5}(0)} u\frac{\p X_{k,\e}}{\p|x|}\,ds + \int\limits_{B_{R_5}(0)\setminus \om_{k,\e}} \nabla u\cdot \nabla X_{k,\e}\,dx
\\
=& \la u\ra \mes_n B_{R_5}(0)\setminus \om_{k,\e} -
\int\limits_{\p B_{R_5}(0)} u_\bot\frac{\p X_{k,\e}}{\p|x|}\,ds + \int\limits_{B_{R_5}(0)\setminus \om_{k,\e}} \nabla u\cdot \nabla X_{k,\e}\,dx.
\end{align*}
Comparing this identity with (\ref{3.30}), we obtain:
\begin{equation*}
\la u \ra =\frac{1}{\mes_n B_{R_3}(0)\setminus\om_{k,\e}} \Bigg(
\int\limits_{\p B_{R_5}(0)} u_\bot\frac{\p X_{k,\e}}{\p|x|}\,ds - \int\limits_{B_{R_5}(0)\setminus \om_{k,\e}} \nabla u\cdot \nabla X_{k,\e}\,dx
\Bigg)
\end{equation*}
and by estimates (\ref{3.22}), (\ref{3.27}), (\ref{3.28}) we see that
\begin{equation*}
|\la u \ra| \leqslant C \|\nabla u\|_{L_2(B_{R_3}(0)\setminus \om_{k,\e})}.
\end{equation*}
Therefore, in view of the first identity in (\ref{3.32}) and estimates (\ref{3.27}), (\ref{3.28}),
\begin{equation}\label{3.33}
\|u\|_{L_2(B_{R_3}(0)\setminus B_{R_5}(0))}\leqslant C \|\nabla u\|_{L_2(B_{R_3}(0)\setminus B_{R_5}(0))}.
\end{equation}
Now we reproduce the proof of Lemma~\ref{lm3.0} and again introduce the lattice $\G_{k,\e}$ and corresponding covering (\ref{3.44}). Estimate (\ref{3.44}) ensures (\ref{3.46}) for cubes located in $B_{R_3}(0)\setminus B_{R_5}(0))$ provided is chosen sufficiently small and fixed. Then as in the proof of Lemma~\ref{lm3.0} we extend estimate (\ref{3.46}) to other cubes and for a tubular neigbourhood of $\p\om_{k,\e}$ getting finally
\begin{equation*}
\|u\|_{L_2(B_{R_5}(0)\setminus \om_{k,\e})}  \leqslant C \|\nabla u\|_{L_2(B_{R_3}(0)\setminus\om_{k,\e})}.
\end{equation*}
This inequality and (\ref{3.33}) yield (\ref{3.26}). The proof is complete.
\end{proof}

\begin{lemma}\label{lm3.3}
Under  Assumption~\ref{A1}
for all $k\in\mathds{M}^\e$ and all $u\in W_2^1(B_{R_3}(0)\setminus\om_k^\e)$ obeying the identity
\begin{equation}\label{3.1}
\int\limits_{B_{\e \eta  R_3}(M_k^\e)\setminus\om_k^\e} u(x)\,dx=0
\end{equation}
the estimates
\begin{align}\label{3.9}
&\|u\|_{L_2(B_{\e\eta R_3}(M_k^\e)\setminus\om_k^\e)}^2\leqslant C\e^2\eta^2 \|\nabla u\|_{L_2(B_{\e   \eta R_3}(M_k^\e)\setminus\om_k^\e)}^2,
\\
&\|u\|_{L_2(\p\om_k^\e)}^2\leqslant C\e\eta \|\nabla u\|_{L_2(B_{\e   \eta R_3}(M_k^\e)\setminus\om_k^\e)}^2,\label{3.17}
\end{align}
hold, where $C$ is a constant independent of the parameters $k$, $\e$, $\eta$ and the function $u$.
\end{lemma}

\begin{proof}
Inequality (\ref{3.9}) is easily obtained by passing to the function $\tilde{u}(\xi):=u(M_k^\e+\e\eta \xi)$ and applying then estimate (\ref{3.26}) to this function as depending on the variable $\xi\in B_{R_3}(0)\setminus \om_{k,\e}$. Then we apply inequality (\ref{3.48}) to the function $\tilde{u}$ and integrate it over $\p\om_{k,\e}$ using the regularity properties of the boundary $\p\om_{k,\e}$. This gives:
\begin{equation}\label{3.34}
\|\tilde{u}\|_{L_2(\p\om_{k,\e})}\leqslant C \|\tilde{u}\|_{W_2^1(B_{R_3}(0)\setminus\om_{k,\e})}\leqslant C \|\nabla\tilde{u}\|_{L_2(B_{R_3}(0)\setminus\om_{k,\e})},
\end{equation}
where $C$ is some constant independent of $\e$, $k$, $\tilde{u}$. Returning back to the function $u$, we arrive at (\ref{3.17}). The proof is complete.
\end{proof}

\begin{lemma}\label{lm3.1}
Under Assumption~\ref{A1}
for all $k\in\mathds{M}_R^\e$ and all $u\in W_2^1(B_{\e R_3}(M_k^\e)\setminus\om_k^\e)$ the estimate
\begin{equation*}
\|u\|_{L_2(\p\om_k^\e)}^2 \leqslant  C\Big(\e\eta \vk\|\nabla u\|_{L_2(B_{\e R_3}(M_k^\e)\setminus\om_k^\e)}^2 +\e^{-1}\eta^{n-1}
\|u\|_{L_2(B_{\e R_3}(M_k^\e)\setminus B_{\e R_2}(M_k^\e))}^2\Big)
\end{equation*}
holds, where  $C$ is a constant independent of the parameters $k$, $\e$, $\eta$ and the function $u$.
\end{lemma}

\begin{proof}
In the case $n\geqslant 3$, this lemma was proved in \cite{MSB21}, see  Lemma~3.3 in this work. Since here we also deal with the case $n=2$, we reproduce the proof of this lemma taking into consideration the mentioned case. We fix $k\in\mathds{M}_R^\e$ and all $u\in W_2^1(B_{\e R_3}(M_k^\e)\setminus\om_k^\e)$ and let $\chi_{3}=\chi_{3}(t)$ be an infinitely differentiable cut-off function equalling to one as $t\leqslant R_2$ and vanishing as $t\geqslant R_5$. Then the function
\begin{equation*}
\tilde{u}(\xi):=\chi_{3}\big(|x-M_k^\e|\e^{-1}\eta^{-1}\big) u(M_k^\e+\e\eta\xi)
\end{equation*}
is an element of $\Ho_2^1(B_{R_3}(0)\setminus\om_{k,\e},\p B_{R_3}(0))$. Applying estimate (\ref{3.34}) and Lemma~\ref{lm3.0}, we obtain:
\begin{equation*}
\|\tilde{u}\|_{L_2(B_{R_3}(0)\setminus \om_{k,\e})} \leqslant C \|\nabla_\xi\tilde{u}\|_{L_2(B_{R_3}(0)\setminus \om_{k,\e})};
\end{equation*}
throughout  the proof by  $C$  we denote various inessential  constants independent of $\e$, $\eta$, $k$, $u$ and $\tilde{u}$.
Returning back to the function $u$, we obtain:
\begin{equation}\label{3.3}
\|u\|_{L_2(\p\om_k^\e)}^2\leqslant C\Big( \e\eta\|\nabla u\|_{L_2(B_{\e\eta R_3}(0)\setminus \om_k^\e)}^2 +  \e^{-1}\eta^{-1}\|u\|_{L_2(B_{\e\eta R_5}(0)\setminus B_{\e\eta R_2}(0)}^2\Big).
\end{equation}

By $\chi_{4}=\chi_{4}(t)$ we denote one more infinitely differentiable cut-off function equalling to one as $t\leqslant R_5$ and vanishing as $t\geqslant R_3$. For $x\in B_{\e\eta R_5}(M_k^\e)$ we have:
\begin{equation*}
u(x)=\int\limits_{\e R_3}^{|x-M_k^\e|}  \frac{\p\ }{\p r} u(x)\chi_{4}(r\e^{-1})\,dr,
\end{equation*}
where $r$ is the radius in the spherical coordinates centered as $M_k^\e$. By the Cauchy-Schwarz inequality for $x\in B_{\e\eta R_5}(M_k^\e)$ we then obtain:
\begin{equation}\label{4.12}
\begin{aligned}
|u(x)|^2\leqslant & C \int\limits_{\e R_3}^{|x-M_k^\e|} \frac{dt}{t^{n-1}}
\int\limits_{\e\eta R_5}^{\e R_3} |\nabla u|^2 t^{n-1}\,dt + C
\e^{-2} \int\limits_{\e R_5}^{\e R_3} \frac{dt}{t^{n-1}} \int\limits_{\e R_5}^{\e R_3} |u|^2 t^{n-1}\,dt
\\
\leqslant & C \e^{-n+2}\eta^{-n+2}\vk \int\limits_{\e\eta R_5}^{\e R_3} |\nabla u|^2 t^{n-1}\,dt + C\e^{-n} \int\limits_{\e R_5}^{\e R_3} |u|^2 t^{n-1}\,dt.
\end{aligned}
\end{equation}
We integrate the obtained inequality over $B_{\e\eta R_5(M_k^\e)}\setminus B_{\e\eta R_2(M_k^\e)}$ and we get:
\begin{equation*}
\|u\|_{L_2(B_{\e\eta R_5(M_k^\e)}\setminus B_{\e\eta R_2(M_k^\e)})}^2
\leqslant C \e^2\eta^2 \vk \|\nabla u\|_{L_2(B_{\e R_3}(M_k^\e)\setminus\om_k^\e)}^2 + C \eta^n \|u\|_{L_2(B_{\e R_3}(M_k^\e)\setminus\om_k^\e)}^2.
\end{equation*}
This estimate and (\ref{3.3}) give the desired estimate in the statement of the lemma.
\end{proof}

\subsection{Proof of Theorem~\ref{th1}}

We begin the proof with checking  the solvability of problem (\ref{2.7}) for appropriate values of $\l$.

\begin{lemma}\label{lm4.1} Under Assumption~\ref{A1} there exists $\l_0\in\mathds{R}$ independent of $\e$ such that for $\RE\l<\l_0$ problem (\ref{2.7}) is   solvable in $\Ho_2^1(\Om^\e,\p\Om\cup\p\tht_D^\e)$ for each $f\in L_2(\Om^\e)$.
\end{lemma}

\begin{proof}
The proof is based on applying standard technique similar to the Browder-Minti theory. Namely, according to the general results in \cite[Ch. V\!I, Sect. 18.4]{Vain}, \cite[Ch. 1, Sect. 1.2$^0$]{Dub}, the solvability (not necessarily unique!) of problem (\ref{2.7}) is ensured by the following conditions:

\begin{enumerate}
\item\label{it1} For all $u,v,w\in \Ho_2^1(\Om^\e,\p\Om\cup\p\tht_D^\e)$ the scalar function $t\mapsto \hf_a(u+tv,w)$ is continuous;

\item\label{it3} The convergence holds:
\begin{equation*}
\frac{\RE \big(\hf_a(u,u)-\l\|u\|_{L_2(\Om^\e)}^2\big)}{
\|u\|_{W_2^1(\Om^\e)}}\to+\infty\qquad\text{as}\qquad \|u\|_{W_2^1(\Om^\e)}\to+\infty.
\end{equation*}
\end{enumerate}

Let us check these conditions. It follows from the definition of the form $\hf_a(u,v)$ that for arbitrary $t_1<t_2$ we have
\begin{equation}\label{4.2}
\hf_a(u+t_2 v,w)-\hf_a(u+t_1 v,w)=(t_2-t_1) \hf_A(v,w) + \big( a^\e(\,\cdot\,,u+t_2 v)-a^\e(\,\cdot\,,u+t_1 v),w\big)_{L_2(\p\tht_R^\e)}.
\end{equation}
Condition (\ref{2.6}) then yields
\begin{equation*}
\left|a^\e(\,\cdot\,,u+t_2 v)-a^\e(\,\cdot\,,u+t_1 v)\right|
\leqslant  a_0^\e|v|(t_2-t_1).
\end{equation*}
This estimate and (\ref{4.2}) imply condition~\ref{it1}.

In order to check  condition 2, we first observe that due to our assumptions on the coefficients $A_{ij}^\e$, $A_j^\e$, $A_0^\e$ the estimate
\begin{equation}\label{4.3}
\RE \hf_A(u,u)\geqslant \frac{3c_0}{4} \|\nabla u\|_{L_2(\Om^\e)}^2 - C \|u\|_{L_2(\Om^\e)}^2
\end{equation}
holds, where $C$ is some absolute constant independent of $\e$ and $u\in W_2^1(\Om^\e)$. It also follows from (\ref{2.11}) and Lemma~\ref{lm3.1}
that
\begin{align*}
\RE \big(a^\e(\,\cdot\,,u),u\big)_{L_2(\p\tht_R^\e)} \geqslant & - c_1
\e\eta^{-n+1} \sum\limits_{k\in\mathds{M}^\e_R} \|u\|_{L_2(\p\om_k^\e)}^2
\\
\geqslant & -C  \Big(\e^2\eta^{-n+2}\vk\|\nabla u\|_{L_2(B_{\e R_3}(M_k^\e)\setminus\om_k^\e)}^2 +
\|u\|_{L_2(B_{\e R_3}(M_k^\e)\setminus B_{\e R_2}(M_k^\e))}^2\Big),
\end{align*}
where $C$ is some fixed constant independent of $\e$, $\eta$ and $u\in W_2^1(\Om^\e)$. According to condition~(\ref{2.14}) we have $\e^2\eta^{-n+2}\vk\to 0$ as $\e\to+0$.
Hence, the above estimate and (\ref{4.3}) yield:
\begin{equation*}
\RE \hf_a(u,u) \geqslant  \frac{c_0}{2} \|\nabla u\|_{L_2(\Om^\e)}^2 - C \|u\|_{L_2(\Om^\e)}^2,
\end{equation*}
where $C$ is some fixed constant independent of $\e$, $\eta$ and $u\in W_2^1(\Om^\e)$. The obtained inequality implies condition~\ref{it3} once we choose $\RE\l\leqslant\l_0< C-1$. The proof is complete.
\end{proof}

We proceed to proving estimates (\ref{2.16}), (\ref{2.17}). We write integral identity (\ref{2.18}) using $u_\e$ as a test function and take then the real part of the obtained relation. The result can be written as
\begin{equation}\label{3.6}
\begin{aligned}
\RE\hf_A(u_\e,u_\e)&+\sum\limits_{k\in\mathds{M}_R^\e\setminus \mathds{M}_{R,0}^\e} \big(a^\e(\,\cdot\,,u_\e),u_\e\big)_{L_2(\p\om_k^\e)}
\\
&+
\sum\limits_{k\in \mathds{M}_{R,0}^\e} \big(a^\e(\,\cdot\,,u_\e),u_\e\big)_{L_2(\p\om_k^\e)}
-\RE\l\|u_\e\|_{L_2(\Om^\e)}^2=\RE(f,u_\e)_{L_2(\Om^\e)}.
\end{aligned}
\end{equation}
Then it follows from inequality (\ref{4.3}), (\ref{2.11}), (\ref{2.25}) and Lemma~\ref{lm3.1} with $\d=\frac{c_0}{4}$ that for the mentioned choice of the constant $\RE\l<\l_0$ the inequality holds:
\begin{equation}\label{4.4}
\frac{c_0}{4} \|\nabla u_\e\|_{L_2(\Om^\e)}^2 + \|u_\e\|_{L_2(\Om^\e)}^2
+ \mu(\e)
\sum\limits_{k\in \mathds{M}_{R,0}^\e} \big(\a_k^\e u_\e,u_\e\big)_{L_2(\p\tht_R^\e)}\leqslant \|f\|_{L_2(\Om^\e)} \|u_\e\|_{L_2(\Om^\e)}.
\end{equation}

Our next step is to estimate the norm $\|u_\e\|_{L_2(\Om^\e)}$. The main idea is to cover the domain $\Om^\e$ by balls $B_{\e R_4}(M_k^\e)$, which is possible owing to Assumption~\ref{A5}, and then to make appropriate local estimates on each such ball intersected with $\Om^\e$.

For $k\in\mathds{M}_R^\e$ we denote
\begin{equation}\label{4.1}
\la u_\e\ra:=\frac{1}{\mes_n B_{\e\eta R_3}(M_k^\e)\setminus\om_k^\e} \int\limits_{B_{\e R_3}(M_k^\e)\setminus \om_k^\e} u_\e\,dx, \qquad
u_\e^\bot:=u_\e-\la u_\e\ra.
\end{equation}
It is clear that the function $u_{\e,k}$ satisfies condition (\ref{3.1}) and hence, by Lemmata~\ref{lm3.3},~\ref{lm3.1},
\begin{align}\label{4.8}
&\|u_\e^\bot\|_{L_2(B_{\e\eta R_3}(M_k^\e)\setminus\om_{k,\e})}^2 \leqslant C\e^2\eta^2 \|\nabla u_\e\|_{L_2(B_{\e\eta R_3}(M_k^\e)\setminus\om_{k,\e})}^2,
\\
&\|u_\e^\bot\|_{L_2(\p\om_{k,\e})}^2 \leqslant C\e\eta \|\nabla u_\e\|_{L_2(B_{\e\eta R_3}(M_k^\e)\setminus\om_{k,\e})}^2,\label{3.12a}
\end{align}
hereinafter in the proof by $C$ we denote various inessential constants  independent of $\e$, $\eta$, $k$, $u_\e$ and spatial variables.

Employing assumptions~(\ref{2.13a}) and Cauchy-Schwarz inequality, we obtain:
\begin{align*}
\int\limits_{\p\om_k^\e} \a_k^\e |u_\e|^2 \,ds= & |\la u_\e\ra|^2 \int\limits_{\p\om_k^\e} \a_k^\e \,ds + 2 \RE \overline{\la u_\e\ra} \int\limits_{\p\om_k^\e} \a_k^\e u_\e^\bot\,ds + \int\limits_{\p\om_k^\e} \a_k^\e |u_\e^\bot|^2 \,ds
\\
\geqslant &  c_3  (\e\eta)^{n-1} |\la u_\e\ra|^2 -2|\la u_\e\ra| \int\limits_{\p\om_k^\e} \a_k^\e |u_\e^\bot|\,ds
\\
\geqslant & \frac{ c_3 }{2}(\e\eta)^{n-1} |\la u_\e\ra|^2 -\frac{2}{c_3}(\e\eta)^{-n+1}
\Bigg(\int\limits_{\p\om_k^\e} \a_k^\e |u_\e^\bot|\,ds\Bigg)^2
\\
\geqslant & \frac{c_3}{2}(\e\eta)^{n-1} |\la u_\e\ra|^2 - \frac{2c_2}{c_3}
\|u_\e^\bot\|_{L_2(\p\om_k^\e)}^2,
\end{align*}
and therefore,
\begin{equation*}
\e^n\eta^n |\la u_\e\ra|^2 \leqslant C \e\eta\Bigg( \int\limits_{\p\om_k^\e} \a_k^\e |u_\e|^2 \,ds + \|u_\e^\bot\|_{L_2(\p\om_k^\e)}^2\Bigg).
\end{equation*}
This estimate and  (\ref{3.12a}),
 (\ref{4.8}) yield:
\begin{equation}\label{4.7}
\begin{aligned}
\|u_\e\|_{L_2(B_{\e\eta R_3}(M_k^\e)\setminus\om_{k,\e})}^2 \leqslant &2
|\la u_\e\ra|^2 \mes_n (B_{\e\eta R_3}(M_k^\e)\setminus\om_{k,\e})
+2
\|u_\e^\bot\|_{L_2(B_{\e\eta R_3}(M_k^\e)\setminus\om_{k,\e})}^2
\\
\leqslant& C   \Bigg(\e\eta \int\limits_{\p\om_k^\e} \a_k^\e |u_\e|^2\,ds + \e^2\eta^2\|\nabla u\|_{L_2(B_{\e R_3}(M_k^\e)\setminus\om_k^\e)}^2 \Bigg).
\end{aligned}
\end{equation}

Let $\chi_{5}=\chi_{5}(t)$ be an infinitely differentiable cut-off function equalling to one as $t>R_3$ and vanishing as $t<R_2$. For $x\in B_{\e R_3}(M_k^\e)\setminus B_{\e\eta R_3}(M_k^\e)$ we have an obvious formula:
\begin{equation*}
u(x)=\int\limits_{\e\eta R_2}^{|x-M_k^\e|} \frac{\p\ }{\p r} u\chi_{5}(r\e^{-1}\eta^{-1})\,dr,
\end{equation*}
where $r$ is the radius in the spherical coordinates centered at $M_k^\e$. Proceeding then as in (\ref{4.12}), we easily obtain:
\begin{equation*}
|u(x)|^2\leqslant C\e^{-n+2}\eta^{-n+2}\vk \int\limits_{\e\eta R_2}^{\e R_3} |\nabla u|^2 r^{n-1}\,dr +
\e^{-n}\eta^{-n} \int\limits_{\e\eta R_2}^{\e\eta R_3} |u|^2 r^{n-1}\,dr.
\end{equation*}
Integrating this estimate over $B_{\e R_3}(M_k^\e)\setminus B_{\e\eta R_3}(M_k^\e)$ and using (\ref{4.7}), we get
\begin{equation}\label{4.9}
\begin{aligned}
\|u_\e\|_{L_2(B_{\e R_3}(M_k^\e)\setminus\om_{k,\e})}^2\leqslant & C \Big(\e^2\eta^{-n+2}\vk \|\nabla u\|_{L_2(B_{\e R_3}(M_k^\e)\setminus\om_{k,\e})}^2+\eta^{-n}\|u\|_{L_2(B_{\e\eta R_3}(M_k^\e)\setminus B_{\e\eta R_2}(M_k^\e))}^2
\Big)
\\
\leqslant & C   \Bigg(\e\eta^{-n+1} \int\limits_{\p\om_k^\e} \a_k^\e |u_\e|^2\,ds + \e^2\eta^{-n+2}\vk\|\nabla u\|_{L_2(B_{\e R_3}(M_k^\e)\setminus\om_k^\e)}^2 \Bigg),\qquad k\in\mathds{M}^\e_{R,0}.
\end{aligned}
\end{equation}
By Lemma~\ref{lm3.4} we have a similar estimate:
\begin{equation}\label{4.6}
\|u_\e\|_{L_2(B_{\e R_3}(M_k^\e)\setminus\om_{k,\e})}^2
\leqslant C \e^2\eta^{-n+2}\vk   \|\nabla u\|_{L_2(B_{\e R_3}(M_k^\e)\setminus\om_k^\e)}^2,\qquad k\in\mathds{M}^\e_D.
\end{equation}

Covering (\ref{2.21}) yields that
\begin{equation}\label{4.5}
\|u_\e\|_{L_2(\Om^\e)}^2\leqslant \sum\limits_{k\in\mathds{M}_{R,0}^\e \cup \mathds{M}_D^\e} \|u_\e\|_{L_2(B_{\e R_4}(M_k^\e)\setminus\tht^\e)}^2.
\end{equation}
The sets $B_{\e R_4}(M_k^\e)\setminus\tht^\e$ not necessarily coincide with $B_{\e R_4}(M_k^\e)\setminus\om_k^\e$ but can have a more complicated shape. The reason is that for a given $k$, the ball $B_{\e R_4}(M_k^\e)$ can have a non-zero intersection with other cavities $\om_j^\e$, $j\ne k$, see Figure~\ref{fig3}. To overcome technical difficulties in local estimates related to possible shapes of the domains $B_{\e R_4}(M_k^\e)\setminus\om_k^\e$, we make an auxiliary continuation of the function $u_\e$. For $k\in\mathds{M}_D^\e$ we simply continue the function $u^\e$ by zero inside the cavities $\om_{k,\e}$:
\begin{equation*}
u_\e(x):=0\quad\text{in}\quad \om_k^\e,\quad k\in\mathds{M}_D^\e.
\end{equation*}
For $k\in\mathds{M}_R^\e$ the aforementioned continuation is made as follows:
\begin{align*}
&u_\e(\tau,s):=\la u_\e \ra + u_\e^\bot(-\tau,s)\chi_{2}(\tau\e^{-1}\eta^{-1}) && \text{for}\quad
x\in\om_k^\e,\quad \dist(x,\p\om_k^\e)\leqslant \e\eta\tau_0,
\\
&u_\e(\tau,s):=\la u_\e \ra && \text{for}\quad
x\in\om_k^\e,\quad \dist(x,\p\om_k^\e)>\e\eta\tau_0,
\end{align*}
with $\la u_\e \ra$ and $u_\e^\bot$ defined in (\ref{4.1}) for the chosen $k$ and $\chi_{2}$ is the cut-off function introduced in the proof of Lemma~\ref{lm3.0}. In view of (\ref{4.1}), the resulting continued function belongs to $W_2^1(B_{\e R_3}(M_k^\e))$ for each $k\in\mathds{M}^\e$ and due to estimate (\ref{3.9}) we have:
\begin{equation}\label{4.14}
\begin{aligned}
&\|u_\e\|_{L_2(\om_k^\e)}^2\leqslant  C \Big(|\la u_\e\ra|^2\mes \om_k^\e + \|u_\e^\bot\|_{L_2(B_{\e\eta R_2(M_k^\e)}\setminus\om_k^\e)}^2\Big)\leqslant C\|u_\e\|_{L_2(B_{\e\eta R_2(M_k^\e)}\setminus\om_k^\e)}^2,
\\
&\|\nabla u_\e\|_{L_2(\om_k^\e)}^2\leqslant  C \Big(\e^{-2}\eta^{-2}\|u_\e^\bot\|_{L_2(B_{\e\eta R_2(M_k^\e)}\setminus\om_k^\e)}^2
 + \|\nabla u_\e\|_{L_2(B_{\e\eta R_2(M_k^\e)}\setminus\om_k^\e)}^2\Big)
\\
&\hphantom{\|\nabla u_\e\|_{L_2(\om_k^\e)}^2} \leqslant C\|\nabla u_\e\|_{L_2(B_{\e\eta R_2(M_k^\e)}\setminus\om_k^\e)}^2.
\end{aligned}
\end{equation}
We also continue $u_\e$ it by zero outside $\p\Om$.
The continued function $u_\e$ is defined on entire $\Om$ and in a fixed neighbourhood of $\Om$  and in view of the above estimates it satisfies
\begin{equation*}
\|u_\e\|_{L_2(\Om)}\leqslant C \|u_\e\|_{L_2(\Om^\e)},\qquad \|\nabla u_\e\|_{L_2(\Om)}\leqslant C \|\nabla u_\e\|_{L_2(\Om^\e)}.
\end{equation*}
Hence, we can rewrite estimate (\ref{4.5}) as
\begin{equation}\label{4.16}
\|u_\e\|_{L_2(\Om^\e)}^2\leqslant \sum\limits_{k\in\mathds{M}_{R,0}^\e \cup \mathds{M}_D^\e} \|u_\e\|_{L_2(B_{\e R_4}(M_k^\e))}^2,
\end{equation}
where for the balls $B_{\e R_4}(M_k^\e)$ intersecting the boundary of $\p\Om$ in the corresponding norms we mean the above made continuation of $u_\e$ by zero outside $\Om$.

For each $k\in\mathds{M}_D^\e\cup\mathds{M}_{R,0}^\e$ the function $u_\e(x)\chi_{5}(|x-M_k^\e|\e^{-1})$ belongs to $W_2^1(B_{\e R_4}(M_k^\e)\setminus\om_k^\e)$,  vanishes  on $B_{\e R_2}(M_k^\e)$  and   coincides with $u_\e$ in the $B_{\e R_4}(M_k^\e)\setminus B_{\e R_3}(M_k^\e)$. Hence, by Lemma~\ref{lm3.4} with $R_3$ replaced by $R_4$, $\eta=1$ and $\om_{k,\e}$ replaced by $B_{R_2}(0)$, for each $k\in\mathds{M}_D^\e\cup\mathds{M}_{R,0}^\e$ we have:
\begin{align*}
\|u_\e\|_{L_2(B_{\e R_4}(M_k^\e)\setminus B_{\e R_3}(M_k^\e))}^2 \leqslant &
\|u_\e\chi_{5}(|\cdot-M_k^\e|\e^{-1})\|_{L_2(B_{\e R_4}(M_k^\e)\setminus B_{\e R_2}(M_k^\e))}^2
\\
\leqslant & C \e^2\|\nabla u_\e\chi_{5}(|\cdot-M_k^\e|\e^{-1})\|_{L_2(B_{\e R_4}(M_k^\e)\setminus B_{\e R_2}(M_k^\e))}^2
\\
\leqslant & C \Big(\e^2\|\nabla u_\e\|_{L_2(B_{\e R_4}(M_k^\e)\setminus B_{\e R_2}(M_k^\e))}^2 + \|u_\e\|_{L_2(B_{\e R_3}(M_k^\e)\setminus B_{\e R_2}(M_k^\e))}^2
\Big).
\end{align*}
Hence, by (\ref{4.9}), (\ref{4.6}), (\ref{4.14}),
\begin{align*}
& \|u_\e\|_{L_2(B_{\e R_4}(M_k^\e))}^2 \leqslant C   \Bigg(\e\eta^{-n+1} \int\limits_{\p\om_k^\e} \a_k^\e |u_\e|^2\,ds + \e^2\eta^{-n+2}\vk\|\nabla u\|_{L_2(B_{\e R_3}(M_k^\e)\setminus\om_k^\e)}^2 \Bigg),\qquad k\in\mathds{M}^\e_{R,0},
\\
&\|u_\e\|_{L_2(B_{\e R_4}(M_k^\e))}^2
\leqslant C \e^2\eta^{-n+2}\vk   \|\nabla u\|_{L_2(B_{\e R_3}(M_k^\e)\setminus\om_k^\e)}^2,\qquad k\in\mathds{M}^\e_D.
\end{align*}
We sum up the obtained estimate over $k\in\mathds{M}_{R,0}^\e\cup\mathds{M}_D^\e$ and by (\ref{4.16}) we find:
\begin{equation}\label{4.10}
\begin{aligned}
\|u_\e\|_{L_2(\Om^\e)}^2\leqslant &
\sum\limits_{k\in\mathds{M}_{R,0}^\e \cup \mathds{M}_D^\e} \|u_\e\|_{L_2(B_{\e R_4}(M_k^\e))}^2
\\
\leqslant & C \e\eta^{-n+1} \sum\limits_{k\in\mathds{M}_{R,0}^\e} \int\limits_{\p\om_k^\e} \a_k^\e |u_\e|^2\,ds + C\e^2\eta^{-n+2}\vk \sum\limits_{k\in\mathds{M}_D^\e\cup \mathds{M}_{R,0}^\e}  \|\nabla u_\e\|_{L_2(B_{\e R_4}(M_k^\e)\setminus\om_k^\e)}^2.
\end{aligned}
\end{equation}

Given a point $x\in\Om^\e$, let us estimate the number
of balls $B_{\e R_4}(M_k^\e)$, $k\in\mathds{M}_{R,0}^\e\cup
\mathds{M}_D^\e$, containing this point; we denote this number by $N_\e(x)$. It is clear that  $N_\e(x)$ is equal to the number of points $M_k^\e$, $k\in\mathds{M}_{R,0}^\e\cup\mathds{M}_D^\e$, such that $\dist(x,M_k^\e)<\e R_4$. By inequality (\ref{2.2}) in Assumption~\ref{A1}, the mutual distances between the points $M_k^\e$ is at least $2 R_3\e$. Associating  each point $M_k^\e$ with a $n$-dimensional cube having a side $2 R_3\e$, we then conclude that $N_\e(x)$ does not exceed the number of such cubes intersecting the ball $B_{\e R_4}(x)$. All such cubes are located inside the ball of the bigger radius $\e(R_4+2\sqrt{n} R_3)$ centered at $x$. Comparing then their volumes, we get:
\begin{equation*}
N_\e(x) \leqslant \frac{\mes B_{\e(R_4+2\sqrt{n}R_3)}(0) }{(2\e R_3)^n} \leqslant \left(\frac{R_4}{2R_3}+\sqrt{n}\right)^n \mes B_1(0),
\end{equation*}
and hence, the number $N_\e(x)$ is bounded uniformly in $\e$ and $x\in\Om^\e$.  This fact allows us to continue estimating in  (\ref{4.10}):
\begin{equation}\label{4.11}
\|u_\e\|_{L_2(\Om^\e)}^2 \leqslant
C \big(\e\eta^{-n+1}\mu^{-1} + \e^2\eta^{-n+2}\vk\big) \Bigg( \mu \sum\limits_{k\in\mathds{M}_{R,0}^\e}  (\a_k^\e u_\e,u_\e)_{L_2(\p\om_k^\e)} +  \|\nabla u_\e\|_{L_2(\Om^\e)}^2
\Bigg).
\end{equation}
Then the right hand of (\ref{4.4}) obeys the inequality
\begin{align*}
\|f\|_{L_2(\Om^\e)} \|u_\e\|_{L_2(\Om^\e)} \leqslant &  C \big(\e\eta^{-n+1}\mu^{-1} + \e^2\eta^{-n+2}\vk\big) \|f\|_{L_2(\Om^\e)}^2
+ \frac{\mu}{2} \sum\limits_{k\in\mathds{M}_{R,0}^\e}  (\a_k^\e u_\e,u_\e)_{L_2(\p\om_k^\e)}
+ \frac{c_0}{8}  \|\nabla u_\e\|_{L_2(\Om^\e)}^2.
\end{align*}
Substituting this estimate into (\ref{4.4}), we finally get:
\begin{equation*}
 \|u\|_{W_2^1(\Om^\e)}^2
+ \mu(\e)
\sum\limits_{k\in \mathds{M}_{R,0}^\e} \big(\a_k^\e u_\e,u_\e\big)_{L_2(\p\tht_R^\e)}\leqslant C \big(\e\eta^{-n+1}\mu^{-1} + \e^2\eta^{-n+2}\vk\big) \|f\|_{L_2(\Om^\e)}^2.
\end{equation*}
 Since $\a_k^\e$ are non-negative, the obtained estimate implies immediately (\ref{2.16}) as well as
\begin{equation*}
\mu \sum\limits_{k\in \mathds{M}_{R,0}^\e} \big(\a_k^\e u_\e,u_\e\big)_{L_2(\p\tht_R^\e)}\leqslant C \big(\e\eta^{-n+1}\mu^{-1} + \e^2\eta^{-n+2}\vk\big) \|f\|_{L_2(\Om^\e)}^2.
\end{equation*}
This estimate and (\ref{2.16}), (\ref{4.11}) then yield (\ref{2.17}).

We proceed to proving that  estimates (\ref{2.16}), (\ref{2.17}) are order sharp. A main idea is to construct particular examples of problem (\ref{2.7}) such that the norms of its solutions have exactly the smallness order stated in estimates (\ref{2.16}), (\ref{2.17}).
In all our examples we let $\Om:=\mathds{R}^n$, $\cL=-\D$, $a^\e(x,u):=u$. The perforation is assumed to be pure periodic, namely, the points $M_k^\e$ are chosen as
as $M_k^\e:=4\e k$, $k\in\mathds{Z}^n$, $k=(k_1,\ldots,k_n)$.
The domains $\om_{k,\e}$ are supposed to be fixed and independent of $k$ and $\e$, and we simply let $\om_{k,\e}:=B_1(0)$. The cavities then are of the form $\om_k^\e:=\{x:\, |x-4\e k|<\e\eta\}$, $k\in\mathds{Z}^n$. The Dirichlet condition is imposed on the boundaries of the cavities $\p\om_p^\e$ as $k_n\geqslant 1$, while the Robin condition is settled on the boundaries $\p\om_p^\e$ for $k_n\leqslant 0$.   It is clear that under the made assumptions, problem (\ref{2.7}) is uniquely solvable for all $f\in L_2(\Om^\e)$ and $\l=0$.

Suppose first that $\eta$
is fixed and is independent of $\e$.
We denote $\square:=\{\xi\in\mathds{R}^n:\, |\xi_i|<2,\ i=1,\ldots,n\}$, where $\xi=(\xi_1,\ldots,\xi_n)$ are   Cartesian coordinates in $\mathds{R}^n$, and consider a boundary value problem
\begin{equation}\label{3.10}
\begin{gathered}
-\D_\xi v_0=1\quad\text{in}\quad \square\setminus B_\eta(0),\qquad v_0=0\quad\text{on}\quad \p B_\eta(0),
\\
v_0\big|_{\xi_i=-2}=v_0\big|_{\xi_i=2},\qquad \frac{\p v_0}{\p\xi_i}\bigg|_{\xi_i=-2}=\frac{\p v_0}{\p\xi_i}\bigg|_{\xi_i=2}.
\end{gathered}
\end{equation}
Owing to the Dirichlet condition on $\p B_\eta(0)$, this problem is obviously uniquely solvable and due to the standard smoothness improving theorems, the solution is infinitely differentiable in $\overline{\square\setminus B_\eta(0)}$. Hereinafter all solutions to various boundary value problems in $\square\setminus B_\eta(0)$ are supposed to be $\square$-periodically continued with the same notations for their continuations.

Let $f_D=f_D(x)$ be a non-zero infinitely differentiable real function with a compact support located in $\{x:\, x_n>5\}$. We introduce one more function:
\begin{equation*}
u_D^\e(x):=\e^2 v_0(x\e^{-1})f_D(x).
\end{equation*}
In view of the properties of the functions $v_0$ and $f_D$, the function $u_D^\e$  is infinitely differentiable, vanishes outside the support of $f_D$ and solves boundary value problem (\ref{2.7}) with
\begin{equation*}
f=f_D - h_D^\e,\qquad h_D^\e(x):=2\e\nabla_\xi v_0(x\e^{-1})\cdot \nabla f_D(x) + \e^2 v_0(x\e^{-1}) \D f_D.
\end{equation*}

Our next step is to calculate  $L_2$- and $W_2^1$-norms  for the functions $u_D^\e$ and $h_D^\e$. In order to do it, we employ the following auxiliary lemma.

\begin{lemma}\label{lm3.5}
Let $v=v(\xi)$ be a  $\square$-periodic function belonging to $L_1(\square\setminus B_{\eta}(0))$   and $h=h(x)$ be a continuously differentiable compactly supported function defined on $\mathds{R}^n$. The identity holds:
\begin{equation*}
\int\limits_{\Om^\e} v(x\e^{-1}) h(x)\,dx= \frac{1}{4^n} \int\limits_{\square\setminus B_{\eta}(0)} v(\xi)\,d\xi \int\limits_{\mathds{R}^d} h(x)\,dx + \e V_\e,
\qquad |V_\e|\leqslant C (\mes\supp h+1)\|v\|_{L_1(\square\setminus B_\eta(0))} \max\limits_{\mathds{R}^d}|\nabla h|,
\end{equation*}
where $C$ is some constant independent of $\e$, $\eta$, $h$, $\supp h$ and $v$.
\end{lemma}

\begin{proof}
We fix an arbitrary $k\in 4\mathds{Z}^n$ and for $x\in \e\square+\e p$ by the Hadamard lemma we have:
\begin{equation*}
h(x)=h(\e k)+h_k(x),\qquad h_k(x)=\sum\limits_{i=1}^{n} (x_i-\e k_i) \int\limits_{0}^{1} \frac{\p h}{\p x_i}(\e k +t (x-\e k))\,dt,\qquad |x_i-\e k_i|\leqslant 2\e.
\end{equation*}
Then by the Cauchy-Schwarz inequality we obtain:
\begin{equation*}
|h_k(x)|\leqslant 2\sqrt{n}\e \max\limits_{\mathds{R}^d} |\nabla h|,\qquad x\in \e k +\e\square.
\end{equation*}
Hence,
\begin{align}
&h(x)=h(\e k)+ h_k(x)  = \frac{1}{4^n\e^n} \int\limits_{\e k + \e\square} (h(x)-h_k(x))\,dx  + h_k(x)=\frac{1}{4^n\e^n} \int\limits_{\e k+ \e\square}  h(x) \,dx  + h^{(k)}(x),\nonumber
\\
&h^{(k)}(x):=h(x)- \frac{1}{4^n\e^n} \int\limits_{\e k+\e\square} h(x)\,dx,\qquad |h^{(k)}(x)|\leqslant 4\sqrt{n}\e \max\limits_{\mathds{R}^d} |\nabla h|.\label{3.36}
\end{align}

Employing the $\square$-periodicity of the function $v$, we represent the integral  in question as
\begin{equation}\label{3.37}
\begin{aligned}
\int\limits_{\Om^\e} v(x\e^{-1}) h(x)\,dx= &\sum\limits_{ \substack{
k\in 4\mathds{Z}^n
\\
(\e k +\e\square)\cap \supp h\ne\emptyset}} \int\limits_{\e k +\e(\square\setminus B_\eta(0))} v(x\e^{-1}) h(x)\,dx
\\
=& \sum\limits_{ \substack{
k\in 4\mathds{Z}^n
\\
(\e k +\e\square)\cap \supp h\ne\emptyset}}\frac{1}{4^n\e^n} \int\limits_{\e\square+\e k} h(x)\,dx
 \int\limits_{\e(\square\setminus B_\eta(0))} v(x\e^{-1})\,dx
 \\
 & +\sum\limits_{ \substack{
p\in 4\mathds{Z}^n
\\
(\e k +\e\square)\cap \supp h\ne\emptyset}} \int\limits_{\e k +\e(\square\setminus B_\eta(0))} v(x\e^{-1}) h^{(k)}(x)\,dx
\\
=& \frac{1}{4^n} \int\limits_{\mathds{R}^d} h(x)\,dx
 \int\limits_{ \square\setminus B_\eta(0) } v(\xi)\,d\xi
  \\
  &+\sum\limits_{ \substack{
k\in 4\mathds{Z}^n
\\
(\e k +\e\square)\cap \supp h\ne\emptyset}} \int\limits_{\e k +\e(\square\setminus B_\eta(0))} v(x\e^{-1}) h^{(k)}(x)\,dx.
\end{aligned}
\end{equation}
We estimate the second term in the right hand side of the obtained identity by means of inequality (\ref{3.36}) and we arrive at the statement of the lemma.
\end{proof}

We employ this lemma to calculate the following norms of the functions $u_D^\e$ and $f_D-h_D^\e$:
\begin{align*}
&\|u_D^\e\|_{L_2(\Om^\e)}^2=\frac{\e^4}{4^n} \|v_0\|_{L_2(\square\setminus B_\eta(0))}^2 \|f_D\|_{L_2(\mathds{R}^n)}^2 + O(\e^5),
\\
&\|\nabla u_D^\e\|_{L_2(\Om^\e)}^2=\frac{\e^2}{4^n} \|\nabla_\xi v_0\|_{L_2(\square\setminus B_\eta(0))}^2 \|f_D\|_{L_2(\mathds{R}^n)}^2 + O(\e^3),
\\
&\|f_D-h_D^\e\|_{L_2(\Om^\e)}^2=\frac{1}{4^n}\|f_D\|_{L_2(\mathds{R}^n)}^2 +O(\e^2).
\end{align*}
Hence,
\begin{equation}\label{3.18}
\frac{\|u_D^\e\|_{L_2(\Om^\e)}}{\|f_D-h_D^\e\|_{L_2(\Om^\e)}} = \e^2 \|v_0\|_{L_2(\square\setminus B_\eta(0))} +O(\e^3), \qquad \frac{\|\nabla u_D^\e\|_{L_2(\Om^\e)}}{\|f_D-h_D^\e\|_{L_2(\Om^\e)}} = \e \|\nabla_\xi v_0\|_{L_2(\square\setminus B_\eta(0))} +O(\e^2).
\end{equation}
These identities show that as $\eta=\eta_0$, the terms $\e \eta^{-\frac{n}{2}+1}(\e)\vk(\e)$ and $\e^2\eta^{-n+2}(\e)\vk(\e)$ in (\ref{2.16}), (\ref{2.17}) are order sharp.

Let us prove that the other terms in (\ref{2.16}), (\ref{2.17}) are also order sharp as $\eta=\eta_0$. We first of all observe that in this case
\begin{equation*}
\e \eta^{-n+1}\mu^{-1}=\e^2 \eta^{-n+1} (\e\mu)^{-1}=\left\{
\begin{aligned}
o&(\e^2)\quad \text{as}\quad \e\mu\to+\infty,\quad \e\to+0,
\\
O&(\e^2)\quad\text{as}\quad  0<C_1\leqslant\e\mu\leqslant C_2\ne0,\quad \e\to+0,
\end{aligned}\right. \qquad C_1,\, C_2=const,
\end{equation*}
and this is why we need to prove the sharpness of the other terms in
(\ref{2.16}), (\ref{2.17}) only in the case $\e\mu\to+0$  as $\e\to+0$.

We consider a boundary value problem
\begin{equation*}
-\D_\xi v_\mu=1\quad\text{in}\quad \square\setminus B_\eta(0),\qquad \frac{\p v_\mu}{\p|\xi|}=\e\mu v_\mu\quad\text{on}\quad \p B_1(0),
\end{equation*}
subject to the periodic boundary conditions as in (\ref{3.10}). This problem is also uniquely solvable. Let us study the behavior of the function $u_\mu$ as $\e\mu\to+0$. By $v_1$ we denote the solution to the problem
\begin{equation}\label{3.20}
-\D_\xi v_1=1\quad\text{in}\quad \square\setminus B_\eta(0),\qquad \frac{\p v_1}{\p|\xi|}=c_4\quad\text{on}\quad \p B_1(0),\qquad c_4:=\frac{\mes \square\setminus B_\eta(0)}{\mes \p  B_\eta(0)},
\end{equation}
subject to the periodic boundary conditions as in (\ref{3.10}). This problem is solvable and there is a unique solution satisfying the identity
\begin{equation}\label{3.21}
\int\limits_{\p B_\eta(0)} v_1\,d\xi=0.
\end{equation}
This solution is infinitely differentiable in $\overline{\square\setminus B_\eta(0)}$. Using this solution,
for each function $u\in C^1(\overline{\square\setminus B_\eta(0)})$
 satisfying the periodic boundary conditions on the lateral boundaries of $\square$ we get the identity
\begin{equation*}
\int\limits_{\square\setminus B_\eta(0)} |u|^2\, d\xi =  -\int\limits_{\square\setminus B_\eta(0)} |u|^2 \D_\xi v_1\, d\xi
= c_4
\int\limits_{\p B_\eta(0)} |u|^2\,ds  + \int\limits_{\square\setminus B_\eta(0)}  \nabla_\xi v_1\cdot\nabla_\xi |u|^2\, d\xi
\end{equation*}
and the estimate
\begin{equation}\label{3.23}
\|u\|_{L_2(\square\setminus B_\eta(0))}^2 \leqslant C \left(
\|u\|_{L_2(\p B_\eta(0))}^2 + \|\nabla_\xi u\|_{L_2(\square\setminus B_\eta(0))}^2\right),
\end{equation}
where $C$ is some constant independent of $u$. Since the space
$C^1(\overline{\square\setminus B_\eta(0)})$ is dense in $W_2^1(\square\setminus B_\eta(0))$, estimate (\ref{3.23}) is also true for all $u\in W_2^1(\square\setminus B_\eta(0))$.

We consider one more boundary value problem
\begin{equation*}
-\D_\xi v_2=0\quad\text{in}\quad \square\setminus B_\eta(0),\qquad \frac{\p v_2}{\p|\xi|}=v_1\quad\text{on}\quad \p B_1(0),
\end{equation*}
subject to the periodic boundary conditions as in (\ref{3.10}). Thanks to condition (\ref{3.21}), this problem is solvable and possesses a unique solution also obeying condition (\ref{3.21}) and being infinitely differentiable in $\overline{\square\setminus B_\eta(0)}$. We denote
\begin{equation*}
\hat{v}_\mu:=v_\mu - \frac{c_4}{\e\mu}
- v_1 - \e\mu v_2.
\end{equation*}
It is easy to see that the introduced function solves the boundary value problem
\begin{equation*}
-\D_\xi \hat{v}_\mu=0\quad\text{in}\quad \square\setminus B_\eta(0),\qquad \frac{\p \hat{v}_\mu}{\p|\xi|}=\e\mu v_\mu - \e^2\mu^2 v_2\quad\text{on}\quad \p B_1(0),
\end{equation*}
subject to the periodic boundary conditions as in (\ref{3.10}). The integral identity corresponding to this problem with $\hat{v}_\mu$ as the test function reads:
\begin{equation*}
\|\nabla \hat{v}_\mu\|_{L_2(\square\setminus B_\eta(0))}^2 + \e\mu \|  \hat{v}_\mu\|_{L_2(\square\setminus B_\eta(0))}^2 = \e^2\mu^2 (v_2, \hat{v}_\mu)_{L_2(B_\eta(0))}.
\end{equation*}
Hence,
\begin{equation*}
\|\nabla \hat{v}_\mu\|_{L_2(\square\setminus B_\eta(0))}^2 + \e\mu \|  \hat{v}_\mu\|_{L_2(\square\setminus B_\eta(0))}^2 \leqslant  \e^2\mu^2 \|v_2\|_{L_2(B_\eta(0))} \|\hat{v}_\mu\|_{L_2(B_\eta(0))}.
\end{equation*}
Then by inequality (\ref{3.23}) we get:
\begin{align*}
&\|\hat{v}_\mu\|_{L_2(\p B_\eta(0))} \leqslant \e\mu \|v_2\|_{L_2(B_\eta(0))},
\\
& \|\nabla\hat{v}_\mu\|_{L_2(\square\setminus B_\eta(0))} \leqslant \e^\frac{3}{2}\mu^\frac{3}{2}\|v_2\|_{L_2(B_\eta(0))},
\\
&\|\hat{v}_\mu\|_{L_2(\square\setminus B_\eta(0))} \leqslant C\e\mu \|v_2\|_{L_2(B_\eta(0))},
\end{align*}
where $C$ is a constant independent of $\e$, $\mu$, $v_2$ and $\hat{v}_\mu$. These inequalities yield the following asymptotic representation for $v_\mu$:
\begin{equation}\label{3.25}
v_\mu=\frac{c_4}{\e\mu}
+ v_1 + O(\e\mu),\quad \e\to+0,
\end{equation}
in $W_2^1(\square\setminus B_\eta(0))$-norm as $\e\mu\to+0$.

Let $f_R=f_R(x)$ be a non-zero infinitely differentiable real function with a compact support located in $\{x:\, x_n<-5\}$. We let:
\begin{equation*}
u_R^\e(x):=\e^2 v_\mu(x\e^{-1})f_R(x).
\end{equation*}
This function  is infinitely differentiable, vanishes outside the support of $f_R$ and solves problem (\ref{2.7}) in the considered particular case with the right hand side
\begin{equation*}
f=f_R - h_R^\e,\qquad h_R^\e(x):=2\e\nabla_\xi v_\mu(x\e^{-1})\cdot \nabla  f_R(x) + \e^2 v_\mu(x\e^{-1}) \D f_R.
\end{equation*}
By Lemma~\ref{lm3.5} and representation
 (\ref{3.25})
we
get:
\begin{equation}\label{3.29}
\begin{aligned}
&\|u_R\|_{L_2(\Om^\e)}= c_4 \e\mu^{-1}
\|f_R\|_{L_2(\mathds{R}^n)} + O(\e^3),
\qquad \|f_R-h_R^\e\|_{L_2(\Om^\e)}=\|f_R\|_{L_2(\mathds{R}^n)}+O(\e\mu^{-1}),
\\
&\|\nabla u_R\|_{L_2(\Om^\e)}=\|c_4\e\mu^{-1}\nabla_x f_R + \e \nabla_\xi v_1(\,\cdot\,\e^{-1}) f_R\|_{L_2(\Om^\e)}
+ O(\e^2\mu)
\end{aligned}
\end{equation}
as $\e\mu\to+0$. It is easy to see that
\begin{equation}\label{3.38}
\begin{aligned}
\|c_4\e\mu^{-1}\nabla_x f_R + \e \nabla_\xi v_1(\,\cdot\,\e^{-1}) f_R\|_{L_2(\Om^\e)}^2=&\e^2\mu^{-2}\|c_4\nabla_x f_R\|_{L_2(\Om^\e)}^2 +\e^2\|\nabla_\xi v_1(\,\cdot\,\e^{-1}) f_R\|_{L_2(\Om^\e)}^2
\\
&+\e^2\mu^{-1}c_4\int\limits_{\Om^\e}  \nabla f_R^2 \cdot \nabla_\xi v_1(\,\cdot\,\e^{-1}) \,dx
\end{aligned}
\end{equation}
and in view of boundary value problem (\ref{3.20}) we can integrate by parts as follows:
\begin{equation}\label{3.39}
\begin{aligned}
\e^2\mu^{-1}\int\limits_{\Om^\e}  \nabla f_R^2 \cdot \nabla_\xi v_1(\,\cdot\,\e^{-1}) \,dx=&-\e^2\mu^{-1} \int\limits_{\p\tht^\e_R} f_R^2 v_1(\,\cdot\,\e^{-1})\,ds -\e\mu^{-1}\int\limits_{\Om^\e}    f_R^2 \D_\xi v_1(\,\cdot\,\e^{-1}) \,dx
\\
=&\e\mu^{-1}\|f_R\|_{L_2(\Om^\e)}^2-\e^2\mu^{-1} \int\limits_{\p\tht^\e_R} f_R^2 v_1(\,\cdot\,\e^{-1})\,ds.
\end{aligned}
\end{equation}
Employing (\ref{3.36}) with $h=f_R^2$ and proceeding as in (\ref{3.37}), we find:
\begin{equation}\label{3.40}
\begin{aligned}
\e^2\mu^{-1} &\int\limits_{\p\tht^\e_R} f_R^2(x) v_1(x\e^{-1})\,ds= \e^2\mu^{-1} \sum\limits_{ \substack{
k\in 4\mathds{Z}^n
\\
(\e k +\e\square)\cap \supp h\ne\emptyset}} \int\limits_{\e k +  \p B_{\e\eta}(0)} f_R^2(x)v_1(x\e^{-1})\,dx
\\
&= \e^2\mu^{-1} \sum\limits_{ \substack{
k\in 4\mathds{Z}^n
\\
(\e k +\e\square)\cap \supp h\ne\emptyset}} \frac{1}{4^n\e^n}\int\limits_{\e k +\e \square} f_R^2(x)\,dx \int\limits_{\e k + \p B_{\e\eta}(0)} v_1(x\e^{-1})\,ds + O(\e^2\mu^{-1})
\end{aligned}
\end{equation}
and in view of (\ref{3.21}) we see that
\begin{equation*}
  \e^2\mu^{-1} \int\limits_{\p\tht^\e_R} f_R^2(x) v_1(x\e^{-1})\,ds= O(\e^2\mu^{-1}).
\end{equation*}
Applying now Lemma~\ref{lm3.5}, by (\ref{3.38}), (\ref{3.39})
we  get:
\begin{equation*}
\|c_4\e\mu^{-1}\nabla_x f_R + \e \nabla_\xi v_1(\,\cdot\,\e^{-1}) f_R\|_{L_2(\Om^\e)}^2=\e\mu^{-1} \|f_R\|_{L_2(\mathds{R}^n)}^2 + O(\e^2\mu^{-1}).
\end{equation*}
This identity and (\ref{3.29}) prove the order sharpness of the terms $\e^\frac{1}{2} \eta^{-\frac{n}{2}+\frac{1}{2}}(\e)\mu^{-\frac{1}{2}}(\e) $ and $\e \eta^{-n+1}(\e)\mu^{-1}(\e)$ in the right hand sides of (\ref{2.16}), (\ref{2.17}) as $\eta$ is independent of $\e$.

We proceed to the case $\eta\to+0$ as $\e\to+0$. We begin with an auxiliary problem
\begin{equation*}
-\D v_3=1\quad\text{in}\quad \square\setminus\{0\},\qquad v_3= G_n(|\xi|)+O(|\xi|^2),\quad \xi\to0,
\end{equation*}
subject to the periodic boundary conditions as in (\ref{3.10}),
where
\begin{equation*}
G_n(t):=\frac{16}{2\pi}\ln t\quad\text{as}\quad n=2,\qquad
G_n(t):=\frac{4^n t^{-n+2}}{(2-n)\mes_{n-1} \p B_1(0)}\quad\text{as}\quad n\geqslant 3,
\end{equation*}
where $\mes_{n-1}(\p B_1(0))$ is the $(n-1)$-dimensional area of the sphere $\p B_1(0)$.
Such problem is uniquely solvable and the solution is infinitely differentiable in $\overline{\square}\setminus\{0\}$.

By $\chi_6=\chi_6(t)$ we denote an infinitely differentiable cut-off function vanishing as $t>2$ and equalling to one as $t<1$. We let $\chi_7(\xi):=\chi_6\big(|\xi|\eta^{-\frac{1}{2}}\big)$ and
\begin{align*}
v_4(\xi):=\big(1-\chi_7(\xi)\big) v_3(\xi) + G_n(|\xi|) \chi_7(\xi) - G_n(\eta),\qquad v_5(\xi):=v_4(\xi) + (\e\mu)^{-1} G_n'(\eta),
\end{align*}
The function $v_4$ satisfies the Dirichlet boundary condition on $\p B_\eta(0)$, while the function $v_5$ does the Robin condition:
\begin{equation*}
\frac{\p v_5}{\p |\xi|}=\e\mu v_5\quad\text{on}\quad\p B_\eta(0).
\end{equation*}
It is also straightforward to confirm that
\begin{align}
-&\D_\xi v_4=-\D_\xi v_5=1+f_\eta(\xi)\quad\text{in}\quad \square\setminus B_\eta(0),\nonumber
\\
&
\begin{aligned}
&f_\eta(\xi):= 2\nabla_\xi \chi_7\cdot \nabla (v_3-G_n)+(v_3-G_n)\D_\xi \chi_7-\chi_7,
\qquad
\|f_\eta\|_{L_2(\square\setminus B_\eta(0))} =O(\eta^\frac{n}{4}),
\\
& \|v_4\|_{L_2(\square\setminus B_\eta(0))}^2 = 4^n G_n^2(\eta) + O\big(\eta^{-n+4}\ln^2\eta\big),
 \\
 &\|v_5\|_{L_2(\square\setminus B_\eta(0))}^2 = 4^n \big((\e\mu)^{-1} G_n'(\eta)-G_n(\eta)\big)^2 + O\big(\eta^{-n+4}\ln^2\eta\big),
\\
& \|\nabla_\xi v_4\|_{L_2(\square\setminus B_\eta(0))}^2 =\|\nabla_\xi v_5\|_{L_2(\square\setminus B_\eta(0))}^2 =c_5 G_n(\eta) + O(\eta^{-n+3}),
\end{aligned}\label{3.35}
\end{align}
where $c_5$ is some fixed non-zero constant independent of $\eta$. We then let
\begin{equation*}
u_D(x):=\e^2 v_4(x\e^{-1}) f_D(x),\qquad u_R(x):=\e^2 v_5(x\e^{-1}) f_R(x),
\end{equation*}
and we see immediately that both these functions solve problem (\ref{2.7}) with the right hand sides
\begin{align*}
&f(x)=f_D(x)+h_D(x,\e), && h_D(x,\e):=f_D(x) f_\eta(x\e^{-1}) - 2\e \nabla_\xi v_4(x\e^{-1}) \cdot \nabla f_D(x) - \e^2  v_4(x\e^{-1}) \D f_D(x),
\\
&f(x)=f_R(x)+h_R(x,\e), && h_R(x,\e):=f_R(x) f_\eta(x\e^{-1}) - 2\e \nabla_\xi v_5(x\e^{-1}) \cdot \nabla f_R(x) - \e^2  v_5(x\e^{-1}) \D f_R(x).
\end{align*}
Applying Lemma~\ref{lm3.5} and using identities (\ref{3.35}), we get:
\begin{equation}\label{3.41}
\begin{aligned}
& \|u_D\|_{L_2(\Om^\e)}^2=\e^4 G_n^2(\eta)\|f_D\|_{L_2(\mathds{R}^d)}^2 + O\big(\e^4\eta^{-n+4}\ln^2\eta+\e^5 G_n^2(\eta)\big),
\\
&\|u_R\|_{L_2(\Om^\e)}^2=\e^4 \big((\e\mu)^{-1} G_n'(\eta)-G_n(\eta)\big)^2\|f_R\|_{L_2(\mathds{R}^d)}^2
\\
&\hphantom{\|u_R\|_{L_2(\Om^\e)}^2=}+ O\big(\e^4\eta^{-n+4}\ln^2\eta+\e^5 \big((\e\mu)^{-1} G_n'(\eta)-G_n(\eta)\big)^2\big),
\\
&\|\nabla u_D\|_{L_2(\Om^\e)}^2=\frac{c_5}{4^n} \e^2 G_n(\eta) \|f_D\|_{L_2(\mathds{R}^d)}^2 + O\big(\e^2\eta^{-n+3} + \e^3 G_n(\eta)\big),
\\
&\|h_D\|_{L_2(\Om^\e)}^2=
O(\eta^{\frac{n}{2}} + \e^2 |G_n(\eta)|),
\qquad \|h_R\|_{L_2(\Om^\e)}^2=
O(\eta^{\frac{n}{2}} + \e^2 G_n(\eta)+\e^2\mu^{-2}\eta^{-2n+2}).
\end{aligned}
\end{equation}

To understand the behavior of the norm $\|\nabla u_R\|_{L_2(\Om^\e)}^2$ as $\e\to+0$, we first observe that in view of the mentioned boundary value problem for $u_R$, it satisfies the integral identity
\begin{equation*}
\|\nabla u_R\|_{L_2(\Om^\e)}^2=(f_R+h_R,u_R)_{L_2(\Om^\e)} - \mu \|u_R\|_{L_2(\p\tht_R^\e)}^2 = (f_R+h_R,u_R)_{L_2(\Om^\e)}  - \e^2\mu^{-1} (G'(\eta))^2 \int\limits_{\p\tht_R^\e} f_R^2(x)\,dx.
\end{equation*}
It follows from (\ref{3.40}) with $v_1$ replaced by $1$ that
\begin{equation*}
\e^2\mu^{-1}(G'(\eta))^2 \int\limits_{\p\tht_R^\e} f_R^2(x)\,dx = \frac{\mes_{n-1} B_1(0)}{4^n} \e\mu^{-1} (G'(\eta))^2 \eta^{n-1} \int\limits_{\mathds{R}^d} f_R^2\,dx + O(\e^2\mu^{-1}\eta^{-n+1}).
\end{equation*}
Then by Lemma~\ref{lm3.5} we obtain:
\begin{align*}
\|\nabla u_R\|_{L_2(\Om^\e)}^2=& \frac{\e^2}{4^n} \|f_R\|_{L_2(\mathds{R}^n)}^2 \int\limits_{\square\setminus B_\eta(0)} v_5(\xi)\,d\xi\, (1+O(\e))
\\
&- \frac{4^n \mes_n B_1(0)}{\mes_{n-1}^2 \p B_1(0)} \e\mu^{-1}\eta^{-n+1} \|f_R\|_{L_2(\mathds{R}^n)}^2  + O(\e^2\mu^{-1}\eta^{-n+1})
\\
=& \left( \frac{4^n}{\mes_{n-1} \p B_1(0)}\left(1 - \frac{1}{n}
\right)\e\mu^{-1} - \e^2 G_n(\eta)\right) \|f_R\|_{L_2(\mathds{R}^n)}^2 + O\big(\e^2\mu^{-1} \eta^{-n+1} +\e^2|G_n(\eta)|\big).
\end{align*}
Employing the obtained identity and (\ref{3.41}) and proceeding as in (\ref{3.18}), we see that  estimates (\ref{2.16}), (\ref{2.17}) are order sharp as $\eta\to+0$.
The proof of Theorem~\ref{th1} is complete.

\subsection{Proof of Theorem~\ref{th2}}

The proof follows the same lines as in the previous subsection and below we just list necessary modifications. Lemma~\ref{lm4.1} remains true and no changes are needed. Identity (\ref{3.6}) and inequality (\ref{4.4}) become
\begin{align}
&\RE\hf_A(u_\e,u_\e)+\sum\limits_{k\in\mathds{M}_R^\e} \big(a^\e(\,\cdot\,,u_\e),u_\e\big)_{L_2(\p\om_k^\e)}
-\RE\l\|u_\e\|_{L_2(\Om^\e)}^2=\RE(f,u_\e)_{L_2(\Om^\e)}, \nonumber
\\
&\frac{c_0}{4} \|\nabla u\|_{L_2(\Om^\e)}^2 + \|u\|_{L_2(\Om^\e)}^2
\leqslant \|f\|_{L_2(\Om^\e)} \|u_\e\|_{L_2(\Om^\e)}. \label{3.7}
\end{align}
Inequality (\ref{4.5}) remains true, just the sum is to be taken over $k\in\mathds{M}_D^\e$. A next step is a modification of estimates (\ref{4.10}), (\ref{4.11}), which is true owing to Lemma~\ref{lm3.4}:
\begin{equation}\label{3.8}
\begin{aligned}
\|u_\e\|_{L_2(\Om^\e)}^2 \leqslant &C \sum\limits_{k\in\mathds{M}_D^\e} \|u_\e\|_{L_2(B_{\e R_4}(M_k^\e))}^2 \leqslant C \e^2 \eta^{-n+2}\vk(\e)  \sum\limits_{k\in\mathds{M}_D^\e} \|\nabla u_\e\|_{L_2(B_{\e R_4}(M_k^\e))}^2
\\
\leqslant& C \e^2 \eta^{-n+2}\vk(\e)  \|\nabla u_\e\|_{L_2(\Om^\e)}^2
\end{aligned}
\end{equation}
with some constants $C$ independent of $\e$, $\eta$ and $u_\e$. This estimate and (\ref{3.7}) then yields (\ref{2.29}). Substituting  this inequality into (\ref{3.8}), we arrive at (\ref{2.30}). The sharpness of estimates (\ref{2.29}), (\ref{2.30}) can be checked by means of the functions $u_D$ introduced in the previous section.  The proof is complete.

\section*{Acknowledgments}

 The authors thank the referee for valuable remarks, which allowed to improve the initial version of the paper. They also  thank V.E.~Bobkov and P.~Freitas for useful discussions on the Cheeger's estimates.

\section*{Funding}

The work is supported by  the Czech Science Foundation within the project 22-18739S.

\section*{Conflict of interest}

The authors declare that they have no conflicts of interest.

\section*{Data availability}

Not applicable in the manuscript as no datasets were generated or analyzed during the current study.

\end{document}